\documentclass[onecolumn,reqno]{amsart}
\usepackage{amsthm}
\usepackage{times,amssymb,amsmath,amsfonts}
\usepackage{graphicx,color}

\newtheorem{theorem}{Theorem}[section]
\newtheorem{lemma}[theorem]{Lemma}
\newtheorem{corollary}[theorem]{Corollary}
\newtheorem{proposition}[theorem]{Proposition}

\theoremstyle{definition}
\newtheorem{definition}[theorem]{Definition}
\newtheorem{example}[theorem]{Example}
\newtheorem{remark}[theorem]{Remark}
\date{\today}

\begin{document}
		\title[]{On polarization of spherical codes and designs}
	\author{P. G. Boyvalenkov} 
\address{ Institute of Mathematics and Informatics, Bulgarian Academy of Sciences \\
8 G Bonchev Str., 1113  Sofia, Bulgaria}
\email{peter@math.bas.bg}

\author{P. D. Dragnev}
\address{ Department of Mathematical Sciences, Purdue University \\
Fort Wayne, IN 46805, USA }
\email{dragnevp@pfw.edu}

\author{D. P. Hardin}
\address{ Center for Constructive Approximation, Department of Mathematics \\
Vanderbilt University, Nashville, TN 37240, USA }
\email{doug.hardin@vanderbilt.edu}

\author{E. B. Saff}
\address{ Center for Constructive Approximation, Department of Mathematics \\
Vanderbilt University, Nashville, TN 37240, USA }
\email{edward.b.saff@vanderbilt.edu}

\author{M. M. Stoyanova} 
\address{ Faculty of Mathematics and Informatics, Sofia University "St. Kliment Ohridski"\\
5 James Bourchier Blvd., 1164 Sofia, Bulgaria}
\email{stoyanova@fmi.uni-sofia.bg}
	 \begin{abstract}  
In this article we investigate the $N$-point min-max and the max-min polarization problems on the sphere for a large class of potentials in $\mathbb{R}^n$. We derive universal lower and upper bounds on the polarization of spherical designs of fixed dimension, strength, and cardinality. The bounds are universal in the sense that they are a convex combination of potential function evaluations with nodes and weights independent of the class of potentials. As a consequence of our lower bounds, we obtain the Fazekas-Levenshtein bounds on the covering radius of spherical designs. Utilizing the existence of spherical designs, our polarization bounds are extended to general configurations. As examples we completely solve the min-max polarization problem for $120$ points on $\mathbb{S}^3$ and show that the $600$-cell is universally optimal for that problem. We also provide alternative methods for solving the max-min polarization problem when the number of points $N$ does not exceed the dimension $n$ and when $N=n+1$. We further show that the cross-polytope has the best max-min polarization constant among all spherical $2$-designs of $N=2n$ points for $n=2,3,4$; for $n\geq 5$, this statement is conditional on a well-known conjecture that the cross-polytope has the best covering radius. This max-min optimality is also established for  all so-called centered codes.
 \end{abstract}
 		\maketitle


\section{Introduction}

Let $\mathbb{S}^{n-1}\subset \mathbb{R}^n$ denote the unit sphere. A collection of distinct points $C=\{x_1,x_2,\ldots,x_N\} \subset   \mathbb{S}^{n-1}$ is called {\em a spherical code}. 
For a function $h:[-1,1] \to (-\infty,+\infty]$, continuous and finite on $[-1,1)$, we consider the {\em discrete $h$-potential associated with $C$}
\[ U_h(x,C):=\sum_{y\in C}h(x \cdot y),\]
where $x \in \mathbb{S}^{n-1}$ is arbitrary. Let 
\begin{equation} \label{max-min}
\mathcal{Q}_h(C):=\inf_{x \in \mathbb{S}^{n-1}} U_h(x,C), \ \ \mathcal{Q}^{(h)}(N):=\sup_{|C|=N} \mathcal{Q}_h(C), 
\end{equation}
be the {\em max-min polarization quantities} and we similarly define the {\em min-max polarization quantities}
\begin{equation} \label{min-max}
\mathcal{R}_h(C):=\sup_{x \in \mathbb{S}^{n-1}} U_h(x,C), \ \ \mathcal{R}^{(h)}(N):=\inf_{|C|=N} \mathcal{R}_h(C). 
\end{equation}
For \eqref{min-max} we additionally suppose that $h$ is continuous and finite on $[-1,1]$. 

For a background discussion of polarization, see \cite[Chapter 14]{BHS}, where potentials on
general compact sets are considered. Such optimization quantities arose in the classical context of \emph{Chebyshev problems} where, for example, one minimizes the maximum norm of a polynomial on a compact set by making an optimal choice for the locations of its zeros.  Polarization problems are intimately connected with so-called covering problems where given a compact set $A$ the goal is to choose $N$ points $\omega_N:=\{x_1, x_2, \ldots, x_N\}$  on $A$ so that the largest
distance of a point of $A$ to $\omega_N$ is minimized. We denote this covering radius by $\delta(A,N).$ In particular,
on $\mathbb{S}^{n-1}$, if $h(t)=(2-2t)^{-s/2}$, with $t=x \cdot y$,  then for each fixed $N$
\begin{equation}\label{AsymptoticRiesz}
\lim_{s \to \infty} \mathcal{Q}^{(h)}(N)^{1/s}=1/\delta(\mathbb{S}^{n-1},N).
\end{equation}
We remark that polarization is also regarded as somewhat dual to the notion of discrete energy
in much the same way that optimal covering problems are somewhat dual to best-packing problems.

Since their introduction in the seminal paper of Delsarte, Goethals and Seidel \cite{DGS} in 1977, spherical designs have been front stage in many optimization problems on $\mathbb{S}^{n-1}$. Many well-known spherical configurations that solve famous and difficult problems turn out to be spherical designs of high strength.

The comprehensive survey \cite{Ban09} (see also \cite{Ban17}) explains the role of spherical designs in algebraic combinatorics, approximation theory (thus, with
connections with analysis and statistics), geometry (in particular sphere packing and covering problems) and optimization theory. 
Strong connections of designs with the maximal spherical codes problem
were revealed and exploited by Levenshtein \cite{Lev83,Lev92} (see the chapter \cite{Lev} for more details). 
Cohn and Kumar
\cite{CK} showed that {\em sharp codes}, that is spherical $(2k-1)$-designs with $k$ different inner products within distinct points of the code, enjoy universal optimality; i.e. they attain minimal 
potential energy among all codes of the same cardinality with respect to all \textcolor{black}{absolutely monotone} potentials. This article's authors obtained in \cite{BDHSS} a universal lower bound (ULB) for energy of all codes that is attained by the same sharp codes, in other words, attained if and only if the Levenshtein bound is attained (see
also \cite{BDHSS-DCC19,BDHSSMathComp}). A long-standing conjecture about the existence of asymptotically optimal spherical designs was resolved in \cite{BRV13} (see also \cite{BRV15}). Last, but not least, it was shown that many results related to spherical designs have their counterparts in other 
important metric spaces, such as the infinite projective spaces, Hamming spaces, and Johnson spaces (see \cite{Lev92,Lev,BDHSS-AMP}).  
General references \cite{CS,EZ,BBIT} describe further interesting features of spherical designs. 

Thus, it is not unexpected that the concept of spherical designs is pivotal in our investigation.

\begin{definition} \label{def-designs-1} A spherical $\tau$-design $C
\subset \mathbb{S}^{n-1}$ is a finite subset of $\mathbb{S}^{n-1}$
such that
\[ \frac{1}{\mu(\mathbb{S}^{n-1})} \int_{\mathbb{S}^{n-1}} f(x) d\mu(x) =
                  \frac{1}{|C|} \sum_{x \in C} f(x) \]
($\mu(x)$ is the surface area measure) holds for all polynomials $f(x)
= f(x_1,x_2,\ldots,x_n)$ of degree at most $\tau$ (i.e., the
average of $f$ over the set $C$ is equal to the average of $f$ over
$\mathbb{S}^{n-1}$). 
\end{definition}

With a spherical $\tau$-design $C$ we shall associate the respective polarization quantities
\begin{equation}\label{PolMinDes}
\mathcal{Q}^{(h)}(n,\tau ,N):=\sup \{ \mathcal{Q}_h(C)\, : \, |C|=N, \,C \ {\rm is \ a }\ \tau{\mbox -}{\rm design} \}
\end{equation}
and
\begin{equation}\label{PolMaxDes}
\mathcal{R}^{(h)}(n,\tau ,N):=\inf \{ \mathcal{R}_h(C)\, : \, |C|=N, \, C \ {\rm is\ a }\ \tau{\mbox -}{\rm design} \}.
\end{equation}
%
%

\begin{definition} \label{Tau_{n,N}} Given a dimension $n$ and cardinality $N$, we define $\tau_{n,N}$ to be the largest $\tau$ for which a spherical $\tau$-design of $N$ points on $\mathbb{S}^{n-1}$ exists. 
\end{definition}

Then the definitions \eqref{max-min}, \eqref{min-max}, \eqref{PolMinDes}, and \eqref{PolMaxDes} immediately imply the bounds
\begin{equation}\label{PULB-PUUB}
\mathcal{Q}^{(h)}(N)\geq \mathcal{Q}^{(h)}(n,\tau_{n,N} ,N),\quad \mathcal{R}^{(h)}(N)\leq \mathcal{R}^{(h)}(n,\tau_{n,N} ,N).
\end{equation}

Borodachov \cite{B} has announced that strongly sharp codes, that is $2k$-designs with $k$ distinct inner products of distinct points within the 
code\footnote{Such codes exist only for $k=1$ (in all dimensions) and for $k=2$ (only in special dimensions).}, 
and sharp antipodal codes, that is, sharp codes containing
at least one antipodal pair, are universally optimal for the min-max polarization problem stated in \eqref{min-max}. Inspired by his work, in this article we shall derive universal lower and upper bounds on max-min and min-max polarization quantities for spherical $\tau$-designs, that is, \textcolor{black}{bounds that are valid for the class of potentials having constant sign of their $(\tau+1)$-st derivative, which includes the class of {\em absolutely monotone potentials of high enough order}}\footnote{\textcolor{black}{A potential is called  {\em absolutely monotone of order $m$} if all of its derivatives up to order $m$ are non-negative and just {\em absolutely monotone} when {\bf all} its derivatives are non-negative.}}. Moreover, these bounds consist of a convex combination of point evaluations of the potential function with nodes and weights independent of the potential. In addition, we will show that the $600$-cell satisfies the same {\em min-max polarization universal optimality} as the strong sharp codes and the sharp antipodal codes. In this regard, we note that the known universally optimal spherical configurations \cite{CK} with respect to the energy of absolutely monotone potentials are the sharp codes and the $600$-cell.

This article is structured as follows. In Section \ref{Sec2} we introduce  the needed preliminaries. A Delsarte-Yudin type linear programing lower bound on min-max and max-min polarization for spherical designs is introduced in Section \ref{LowerBound} with the main results  being universal lower bounds given in Theorem \ref{PULB} \textcolor{black} {and Theorem \ref{PULB_Negative}} and conditions for attaining the max-min polarization bound (see Corollary \ref{attaining-cor-1}). Examples \ref{Cube} and \ref{Lower24cell} provide an application of the results in the section and establish that our bound \eqref{PolarizationULB} in Theorem \ref{PULB} and the Fazekash-Levenshtein bound are simultaneously 
attained for the cube on $\mathbb{S}^2$ and the $24$-cell $D_4$ on $\mathbb{S}^3$, respectively. Proposition \ref{Negative_Cross} shows that the cross-polytope attains the bound \eqref{PolarizationULB_Negative} in Theorem \ref{PULB_Negative}. Section \ref{UpperBound}, after
setting a Delsarte-Yudin type upper bound, introduces the notion of  degree $m$ positive definiteness of signed measures and presents the universal upper bound on min-max and max-min polarization for spherical designs in Theorem \ref{PUUB}. 
In Section 5 we establish that the $600$-cell (the unique $11$-design of $120$ points on $\mathbb{S}^3$, see \cite{A2, Boy}) is universally optimal for the min-max polarization problem.  Further applications and illustrations of the theory in Section 6 include an alternative method of solving the max-min polarization property of the cases $N\leq n$ (the optimal configuration being any spherical $1$-design) and $N=n+1$ (the optimal configuration being the regular simplex). We also prove that the cross-polytope has the best max-min polarization constant among all spherical $2$-designs of $N=2n$ points for $n=2,3,4$; for $n\geq 5$, this statement is conditional on a well-known conjecture that the cross-polytope has the best covering radius. In addition, we establish the max-min polarization optimality of the cross-polytope for all so-called centered codes.

\section{Preliminaries}\label{Sec2}

\subsection{Gegenbauer and adjacent polynomials}
We remind the reader that the {\em Gegenbauer polynomials} $P_k^{(n)} (t)=P_k^{(0,0)}(t)$ are Jacobi\footnote{The Jacobi polynomials 
$P_k^{\alpha,\beta}(t)$ are orthogonal on $[-1,1]$ with respect to a weight function $(1-t)^\alpha (1+t)^\beta$.} polynomials $P_k^{\alpha,\beta}(t)$ with 
$\alpha=\beta=(n-3)/2$ normalized so that $P_k^{(n)} (1)=1$. Namely, the Gegenbauer polynomials $P_i^{(n)}(t)$ 
are orthogonal on $[-1,1]$ with respect to the measure 
\[d\mu_n(t):=\gamma_n (1-t^2)^{(n-3)/2}\, dt , \] 
where the normalization constant $\gamma_n$ is chosen so that $\gamma_n \int_{-1}^1 (1-t^2)^{(n-3)/2}\, dt=1$. 

We shall utilize also the {\em adjacent Gegenbauer polynomials} $P_k^{(a,b)} (t)$, $a,b\in\{0,1\}$, namely Jacobi polynomials with 
$\alpha=a+(n-3)/2$ and $\beta=b+(n-3)/2$ similarly normalized by $P_k^{(a,b)} (1)=1$. 

Let 
\begin{equation} \label{geg-exp} 
f(t)=\sum_{i=0}^{\deg(f)} f_i P_i^{(n)}(t) 
\end{equation}
be a real polynomial expanded in terms of the Gegenbauer polynomials $P_i^{(n)}(t)$. The Gegenbauer coefficients $f_i$ are given by
\[f_i := \gamma_n \int_{-1}^1 f(t) P_i^{(n)}(t)(1-t^2)^{(n-3)/2}\, dt, \quad i=0,\dots, \deg(f). \]

\subsection{Another (crucial) definition for spherical designs}
The following equivalent definition of spherical designs facilitates our approach to the polarization problem (see \cite{DGS}, \cite[Equation (1.10)]{FL}). 

\begin{definition} \label{def-des} A code
$C \subset \mathbb{S}^{n-1}$ is a spherical $\tau$-design if and
only if for any point $x \in \mathbb{S}^{n-1}$ and any real
polynomial $f(t)$ of degree at most $\tau$, the equality
\begin{equation}
\label{defin_f}
U_f(x,C)=\sum_{y \in C}f(x \cdot y) = f_0|C|
\end{equation}
holds, where $f_0=\gamma_n \int_{-1}^1 f(t)(1-t^2)^{(n-3)/2}\, dt$ is the constant coefficient in the Gegenbauer expansion  \eqref{geg-exp} of $f$.
\end{definition}

Note that \eqref{defin_f} asserts that the polarization $f$-potential for $C$ is constant on ${\mathbb S}^{n-1}$. This fact serves as the foundation in obtaining lower and upper 
linear programming (LP) bounds for the polarization of spherical designs. 

\subsection{Moments of spherical codes}

 Recall that given a spherical code $C$, its $i$-th moment is defined as
\begin{equation}\label{Moments}
M_i^n(C):=\sum_{x,y\in C} P_i^{(n)} (x\cdot y).
\end{equation}
We remark that one of the alternative definitions for a spherical design of strength  $\tau$ is that all the moments  $M_i^n(C)$ vanish for $i=1,\dots,\tau$.
The following lemma from \cite{BHS} shows the relation between the moments and the spherical harmonics.

\begin{lemma}\label{BHSLemma} (\cite[Lemma 5.2.2]{BHS}) Let $C\subset\mathbb{S}^{n-1}$, $|C|=N$. Let $\mathbb{H}_k^n$ be the subspace of spherical harmonics of degree $k\in \mathbb{N}$ and $Z(n,k)$ denote the dimension of this subspace. Then the following are equivalent:
\begin{itemize}
\item[(a)] The moment $M_k^n(C) =0$.
\item[(b)] For any orthogonal basis  $\{ Y_{kj}(x) \}_{j=1}^{Z(n,k)}$ of \, $\mathbb{H}_k^n$ 
\[\sum_{x\in C} Y_{k,j}(x)=0, \quad 1\leq j\leq Z(n,k).\]
\item[(c)] For any $x\in \mathbb{S}^{n-1}$,
\[ \sum_{y\in C} P_k^{(n)}(x\cdot y)=0.\]
\end{itemize}
\end{lemma}

Lemma \ref{BHSLemma} allows us to extend Definition \ref{def-des}. Let $C \subset \mathbb{S}^{n-1}$ be a spherical code and $I(C) \subset \mathbb{N}$ be the collection of indexes for which $ M_i^n(C)=0$ (also referred to as the {\em index set associated with the code $C$}).
\begin{proposition} \label{prop23}
Let $C \subset \mathbb{S}^{n-1}$ be a spherical code with $I(C)\not= \emptyset$. If the polynomial $f$ has $f_i=0$ 
for all $i \in \mathbb{N} \setminus I(C)$, then \eqref{defin_f} holds true.  
\end{proposition} 
We will use this property in Section 5 when we consider the 600-cell. 

\subsection{Delsarte-Goethals-Seidel bound}
The cardinality of spherical $\tau$-designs is bounded from below by the Fisher-type bound 
(see Delsarte, Goethals, and Seidel \cite[Theorems 5.11, 5.12]{DGS}) as follows. If $C\subset\mathbb{S}^{n-1}$ is a $\tau$-design, 
$\tau=2k-1+\epsilon$, $k \in \mathbb{N}$, $\epsilon \in \{0,1\}$, then 
\begin{equation}
\label{DGS-bound}
|C| \geq D(n,\tau):={n+k-2+\epsilon \choose n-1}+{n+k-2 \choose n-1}.
\nonumber
\end{equation}
The first few bounds are $D(n,1)=2$ for 1-designs, $D(n,2)=n+1$ for $2$-designs, $D(n,3)=2n$ for 3-designs 
(all these are attained -- by a pair of antipodal points; a regular $n$-dimensional simplex; and a cross-polytope in $n$-dimensions, respectively), $D(n,4)=n(n+3)/2$ for $4$-designs, and $D(n,5)=n(n+1)$ for $5-$designs.  

Seymour and Zaslavsky \cite{SZ} showed that for fixed $n$ and $\tau$, there exist spherical designs of all large enough cardinalities. 
In the asymptotic process when the dimension $n$ is fixed and the strength $\tau$ tends to infinity, 
Bondarenko, Radchenko and Viazovska \cite{BRV13,BRV15} proved the existence of spherical $\tau$-designs with cardinalities $N \geq C_n \tau^{n-1}$, where $C_n$ depends only on the dimension $n$, thus giving the optimal order of the minimal
designs, as mentioned above.  

\subsection{Covering radius and Fazekas-Levenshtein bound}

For a spherical code $C\subset \mathbb{S}^{n-1}$ define the {\em covering radius of $C$} as 
$$\rho(C):=\max_{x\in \mathbb{S}^{n-1}} \min_{y\in C}\| x-y\|.$$
The {\em best covering radius for spherical codes} of cardinality $N$ is defined as (see text above \eqref{AsymptoticRiesz})
$$\rho(N):=\delta(\mathbb{S}^{n-1},N)=\min_C \{ \rho(C)\, : \, |C|=N \}  ,$$  
and for {\em spherical  $\tau$-designs} of cardinality $N$ (when they exist) as 
$$\rho(\tau,N):=\min_C \{ \rho(C)\, :\, C {\rm \ is\ a\ } \tau\mbox{-design},  |C|=N\}.$$  

For our purposes it is beneficial to re-phrase these quantities in terms of inner products. Namely, given a code $C \subset \mathbb{S}^{n-1}$ 
of cardinality $N$ and a point $x$ on $\mathbb{S}^{n-1}$, denote the set of inner products between $x$ and the points of the code with $T(x,C)$ and define
\begin{equation}\label{CoveringRadius}
s_C:=\min_x \max T(x,C)=1-\frac{\rho(C)^2}{2}, \quad s_N:=\max_{|C|=N} s_C=1-\frac{\rho(N)^2}{2}.\end{equation}
For the case of spherical $\tau$-designs $C$ of cardinality $|C|=N$ we shall use the notations
\begin{equation}\label{CoveringRadiusDes}
\ell_{\tau,N}:=\min_C \{s_C  : C {\rm \ is \ } \tau\mbox{-design} \}, \quad s_{\tau,N}:=\max_{C} \{s_C : C {\rm \ is \ } \tau\mbox{-design}\}
\end{equation}
(note that $s_{\tau,N}=1-\rho(\tau,N)^2/2$). 

Linear programming bounds for the covering radius of spherical designs were obtained by Fazekas and
Levenshtein \cite[Theorem 2]{FL} in the setting of polynomial metric spaces. The Fazekas-Levenshtein bound 
states that if $C$ is a $\tau$-design, $\tau=2k-1+\epsilon$, $\epsilon \in \{0,1\}$, then $s_C \geq t_k^{0,\epsilon}$, i.e.
\begin{equation}
\label{FL_bound}
\ell_{\tau,N} \geq t_k^{0,\epsilon},
\end{equation}
where $t_k^{0,\epsilon}$ is the largest zero of the (possibly adjacent) polynomial $P^{(0,\epsilon)}(t)$. 
The first few bounds are $0$ for 1-designs ($N \geq 2$), $1/n$ for $2$-designs ($N \geq n+1$), $1/\sqrt{n}$ for $3$-designs
($N \geq 2n$), $\frac{1+\sqrt{n+3}}{n+2}$ for $4$-designs ($N \geq n(n+3)/2$), and $\sqrt{3/(n+2)}$ for $5$-designs ($N \geq n(n+1)$).

Note that the right hand side of the bound \eqref{FL_bound} does not depend on the cardinalities of the designs under consideration. 
Lower and upper linear programming bounds on the covering radius of spherical $\tau$-designs that take into account the cardinalities were derived in \cite{BS21}. 

The Fazekas-Levenshtein bound will be derived as a consequence of our Theorem \ref{PULB} below. Also, we show that the 
cube in three dimensions and the $24$-cell in four dimensions attain the bound \eqref{FL_bound} for $(\tau,N)=(3,8)$ and $(5,24)$, respectively,
two examples that seem to be unnoticed so far (see Examples \ref{Cube} and  \ref{Lower24cell}).

\section{Universal lower bounds for polarization}\label{LowerBound}

\begin{definition} \label{LowerClass} 
Let $\tau$ be a positive integer and $h$ be a potential function. Denote by $\mathcal{L}(n,\tau,h)$ 
the class of {\em lower admissible polynomials} $f(t)$ such that 
\begin{itemize}
\item[(A1)] $\deg(f) \leq \tau$;
\item[(A2)] $f(t) \leq h(t)$ for every $t \in [-1,1]$.
\end{itemize}
\end{definition}

Utilizing this definition  and \eqref{defin_f} we derive the following Delsarte-Yudin type lower bound on the polarization potential of $h(t)$.

\begin{proposition} \label{prop32}
Let $h(t)$ be a potential function, $\tau$ be a natural number, and $f \in {\mathcal L}(n,\tau, h)$ be a lower admissible polynomial. Then for all spherical $\tau$-designs $C \subset \mathbb{S}^{n-1}$ the following lower bound holds:
\begin{equation}\label{lower-bound}
U_h(x,C) \geq  U_f(x,C) = f_0|C|, \quad x\in \mathbb{S}^{n-1}.
\end{equation}
Consequently,
\begin{equation}\label{DY_LB}
\mathcal{Q}_h(C) \geq \max_{f\in {\mathcal L}(n,\tau,h)} \{ f_0|C| \}. 
\end{equation}
\end{proposition}

\begin{proof}
This is immediate from Definition \ref{def-des}.
\end{proof}

The following corollary is a trivial consequence of \eqref{DY_LB}.

\begin{corollary} Suppose that the collection of spherical $\tau$-designs $C \subset \mathbb{S}^{n-1}$ of cardinality $|C|=N$ is non-empty. Then
\begin{equation} \label{Polarization_LB1}\mathcal{Q}^{(h)}(n,\tau,N) \geq \max_{f\in {\mathcal L}(n,\tau,h)} \{ f_0N \}.
\nonumber
\end{equation}
\end{corollary}

This gives rise to the following linear program:
\begin{align}\label{LP_ULB}
{\rm given}: & \quad n, \ \tau, \ N, \ h \nonumber \\
{\rm maximize}: & \quad f_0  \\
{\rm subject \ to:} &\quad  f \in {\mathcal L}(n,\tau,h) \nonumber
\end{align}

Theorems \ref{PULB} and \ref{PULB_Negative} below are two of our first main results that solve the linear program \eqref{LP_ULB} and establishe a universal lower bound on the max-min and the min-max polarization quantities of spherical designs, which we refer to as a polarization ULB or simply PULB.

\begin{theorem}\label{PULB}  Assume there exists a spherical $\tau$-design of cardinality $N$ on $\mathbb{S}^{n-1}$, where $\tau=:2k-1+\epsilon$, $\epsilon\in \{0,1\}$, and that the potential $h$ is continuous on $[-1,1]$, finite on $(-1,1)$ and that $h^{(\tau+1)}(t)\geq 0$ on $(-1,1)$. Then the following {\em max-min and min-max Polarization ULB} hold:
\begin{equation}\label{PolarizationULB}
\mathcal{Q}^{(h)}(n,\tau,N) \geq N \sum_{i=1-\epsilon}^{k} \rho_i h(\alpha_i), 
\end{equation}
and 
\begin{equation}\label{PolarizationULB_R}
\mathcal{R}^{(h)}(n,\tau,N) \geq N \sum_{i=1-\epsilon}^{k} \rho_i h(\alpha_i), 
\end{equation}
where $\{\alpha_i\}_{i=1-\epsilon}^k$ are the roots of the (possibly adjacent) Gegenbauer polynomial $(1+t)^\epsilon P_k^{(0,\epsilon)}(t)$ and the corresponding weights $\{ \rho_i \}_{i=1-\epsilon}^k$ are positive with sum equal to $1$, and are given by 
\begin{equation} \label{RhoWeights}
\rho_i=\gamma_n\int_{-1}^1 \ell_i (t) (1-t^2)^{(n-3)/2}\, dt=\gamma_n\int_{-1}^1 \ell_i^2 (t) (1-t^2)^{(n-3)/2}\, dt,
\end{equation} 
where $\ell_i(t)$ are the Lagrange basic polynomials associated with the nodes $\{ \alpha_i \}_{i=1-\epsilon}^k$. Moreover, the bound \eqref{PolarizationULB} 
is the best that can be attained by polynomials from the set ${\mathcal L}(n,\tau,h)$.
\end{theorem}

\begin{proof} The bound \eqref{PolarizationULB_R} is an easy consequence of \eqref{lower-bound} , so we shall focus on \eqref{PolarizationULB}.

We first consider the case when $\epsilon=0$; i.e., $\tau=2k-1$ is odd. Recall the Gauss-Jacobi quadrature with Jacobi parameters $\alpha=\beta=(n-3)/2$. Namely, given a positive integer $k$ if we select as nodes $\alpha_i$, $i=1,\dots,k$, to be the $k$ distinct roots of $P_k^{(n)}(t)$ in $(-1,1)$, we can find 
positive weights $\rho_i$, $i=1,\dots,k$, such that the quadrature 
\begin{equation} \label{GJQ}
\gamma_n \int_{-1}^1 f(t) (1-t^2)^{\frac{n-3}{2}} \, dt=\sum_{i=1}^k \rho_i f(\alpha_i) 
\end{equation}
is exact for all polynomials of degree at most $2k-1$. Substituting with the Lagrange basic polynomials $\ell_i (t)$, such that $\ell_i(\alpha_j)=\delta_{ij}$, where $\delta_{ij}$ is the Kronecker  delta symbol, we obtain the quadrature weights on the right side of \eqref{PolarizationULB}. By a standard argument the positivity of the $\rho_i$ can be seen on substituting $f(t)=\ell_i^2(t)$ in \eqref{GJQ}, and on setting $f(t)=1$ we derive that the sum of the weights equals $1$. Hence, the bound is a convex combination of point evaluations at the quadrature nodes, which do not depend on the underlying potential function. 

We next define the Hermite interpolation polynomial $H_{2k-1}(t;h)$ such that $ H_{2k-1}(\alpha_i;h)=h(\alpha_i)$ and $H_{2k-1}^\prime(\alpha_i;h)=h^\prime (\alpha_i)$, $i=1,\dots,k$. It follows that $\deg(H_{2k-1}(t;h)) \leq 2k-1=\tau$. 
Utilizing the error form of the Hermite interpolation formula (see \cite[Lemma 2.1]{CK}) and the fact that $h$ has non-negative derivative of order $2k$ we conclude that for any $t\in (-1,1)$, there exists a $\xi\in (-1,1)$ such that
\begin{equation}\label{HermiteError}
h(t)-  H_{2k-1}(t;h) = \frac{h^{(2k)}(\xi)}{(2k)!}(t-\alpha_1)^2 \dots (t-\alpha_k)^2\geq 0.\end{equation}
Therefore, $H_{2k-1}\in {\mathcal L}(n,\tau, h)$ and hence,
\begin{eqnarray*}
{\mathcal Q}^{(h)}(n,\tau,N) &\geq& N\gamma_n \int_{-1}^1 H_{2k-1}(t;h) (1-t^2)^{(n-3/2}\, dt \\
                           &=& N \sum_{i=1}^k \rho_i H_{2k-1}(\alpha_i;h)=N\sum_{i=1}^k \rho_i h(\alpha_i).
\end{eqnarray*}

To see that the bound is optimal among polynomials in ${\mathcal L}(n,\tau,h)$, suppose that $f$ is an arbitrary polynomial in the class. We then have 
\[ Nf_0=N\gamma_n \int_{-1}^1 f(t) (1-t^2)^{(n-3/2}\, dt=N \sum_{i=1}^k \rho_i f(\alpha_i)\leq N \sum_{i=1}^k \rho_i h(\alpha_i).\]
For equality to hold, we need $f(\alpha_i)=h(\alpha_i)$ for $i=1,\dots,k$, which along with $f(t)\leq h(t)$ implies that $f^\prime (\alpha_i)=h^\prime (\alpha_i )$. Since the Hermite interpolation polynomial  $H_{2k-1}(t;h)$ is unique, this is the only polynomial attaining the bound \eqref{PolarizationULB}. 

When $\epsilon=1$, that is, $\tau=2k$ is even, the proof above is modified as follows. We take the nodes $\{\alpha_i \}_{i=1}^k$ to be the roots of $P_k^{(0,1)}(t)$. Setting $\alpha_0=-1$, we define the Lagrange basic polynomials $\ell_i (t)$, $i=0,1,\dots, k$, and let $\rho_i:=\gamma_n \int_{-1}^1 \ell_i (t)(1-t^2)^{(n-3)/2}\, dt$, $i=0,\dots,k$. With this definition the quadrature rule
\begin{equation} \label{QR_epsilon}
\gamma_n \int_{-1}^1 f(t) (1-t^2)^{\frac{n-3}{2}} \, dt=\sum_{i=0}^k \rho_i f(\alpha_i) 
\end{equation}
is exact for all polynomials up to degree $k$. This quadrature rule is actually exact for all polynomials of degree at most $2k$. Indeed, any polynomial $f(t)$ of degree at most $2k$ can be presented as 
$$f(t)=(1+t)P_k^{(0,1)}(t)q(t)+p(t),$$ 
for some polynomials $q(t)$ and $p(t)$ of degree at most $k-1$ and $k$ respectively. Then we use the orthogonality of $P_k^{(0,1)}$ and $q$ to obtain
\begin{eqnarray*}
\gamma_n \int_{-1}^1 f(t) (1-t^2)^{\frac{n-3}{2}} \, dt &=& \gamma_n \int_{-1}^1  \left[ (1+t)P_k^{(0,1)}(t)q(t)+p(t)\right] (1-t^2)^{\frac{n-3}{2}} \, dt \\
&=& \gamma_n \int_{-1}^1  p(t)(1-t^2)^{\frac{n-3}{2}} \, dt ,
\end{eqnarray*}
while
\begin{eqnarray*}
\sum_{i=0}^k \rho_i f(\alpha_i) &=& \sum_{i=0}^k \rho_i \left[ (1+\alpha_i) P_k^{(0,1)}(\alpha_i) q(\alpha_i)+p(\alpha_i) \right] \\
&=&  \sum_{i=0}^k \rho_i p(\alpha_i) = \gamma_n \int_{-1}^1  p(t)(1-t^2)^{\frac{n-3}{2}} \, dt ,
\end{eqnarray*}
where we used that \eqref{QR_epsilon} holds for $p$. This completes the proof that quadrature \eqref{QR_epsilon} is valid for polynomials of degree up to $2k$. 

Let $H_{2k}(t;h)$ be  the polynomial of degree at most $2k$ that interpolates $h(t)$ at $\{ \alpha_i \}_{i=0}^k$ and its derivative $h^\prime (t)$ at the interior nodes $\{\alpha_i \}_{i=1}^k$. The Hermite error formula takes the form
\begin{equation}\label{HermiteError1} 
h(t)-  H_{2k}(t;h) = \frac{h^{(2k+1)}(\xi)}{(2k+1)!}(t+1)(t-\alpha_1)^2 \dots (t-\alpha_k)^2\geq 0,
\end{equation}
so we conclude that $H_{2k}(t)\in {\mathcal L}(n,\tau,h)$. The remaining part of  the proof of \eqref{PolarizationULB} is the same as for odd $\tau$.
\end{proof}

As a byproduct of Theorem \ref{PULB} and the asymptotic formula \eqref{AsymptoticRiesz}, we obtain a simple proof of the Fazekas-Levenshtein bound \eqref{FL_bound}.
\begin{corollary}
Let there be a spherical $\tau$-design $C$, $|C|=N$, where $\tau=2k-1+\epsilon$, $\epsilon\in \{0,1\}$. Then $\ell_{\tau,N}\geq \alpha_k=t_k^{0,\epsilon}$.
\end{corollary}
\begin{proof} 
Suppose to the contrary that $\ell_{\tau_{n,N},N}< \alpha_k$. Then there exists a spherical $\tau$-design $C$ of cardinality $|C|=N$ such that $s_C<\alpha_k$ and 
\[ \sum_{y\in C}h(x\cdot y)\geq N\sum_{i=1-\epsilon}^{k} \rho_i h(\alpha_i)\]
for all absolutely monotone potentials and, in particular, for the Riesz potentials $h_m (t)=(2-2t)^{-m/2}$ (as their derivative of order $(\tau_{n,N}+1)$ is non-negative). The monotonicity of $h_m (t)$ yields
\[ \frac{N}{(2-2s_C)^{m/2}} \geq \sum_{y\in C}h_m (x\cdot y) \geq N\sum_{i=1-\epsilon}^{k} \rho_i h_m(\alpha_i) \geq \frac{N\rho_k}{(2-2\alpha_k)^{m/2}}.\]
Taking an $m$-th root and letting $m\to \infty$ we derive a contradiction.
\end{proof}

\begin{remark} We note that the assumption $h^{(\tau+1)}(t)\geq 0$ on the potentials is significantly weaker than the standard assumption of absolutely monotone potentials when dealing with universal results for the minimal energy problem. Moreover, we are able to modify our analysis to determine analogous PULBs for spherical $\tau$-designs with $h^{(\tau+1)}(t)\leq 0$ as the following theorem shows.
\end{remark}

\begin{theorem}\label{PULB_Negative}  \textcolor{black}{
Assume there exists a spherical $\tau$-design of cardinality $N$ on $\mathbb{S}^{n-1}$, where $\tau=:2k-1+\epsilon$, $\epsilon\in \{0,1\}$, and that the potential $h$ is continuous on $[-1,1]$, finite on $(-1,1)$ and that $h^{(\tau+1)}(t)\leq 0$ on $(-1,1)$.  Then the following {\em max-min and min-max Polarization ULB} hold:
\begin{equation}\label{PolarizationULB_Negative}
\mathcal{Q}^{(h)}(n,\tau,N) \geq N\cdot \sum_{i=\epsilon}^{k+\epsilon} \rho_i h(\alpha_i), 
\end{equation}
and 
\begin{equation}\label{PolarizationULB_R_Negative}
\mathcal{R}^{(h)}(n,\tau,N) \geq N\cdot \sum_{i=\epsilon}^{k+\epsilon} \rho_i h(\alpha_i),
\end{equation}
where when $\epsilon=0$ the nodes are $\alpha_0=-1$; the roots $\{\alpha_i\}_{i=1}^{k-1}$  of the adjacent Gegenbauer polynomial $P_{k-1}^{(1,1)}(t)$; and $\alpha_k=1$; and when $\epsilon=1$ the nodes are the roots $\{\alpha_i\}_{i=1}^k$ of the adjacent Gegenbauer polynomials $P_k^{(1,0)}(t)$; and $\alpha_{k+1}=1$. The associated positive weights $\{ \rho_i \}$ with sum equal to $1$ are given by formulas  analogous to \eqref{RhoWeights}. Moreover, the bound \eqref{PolarizationULB_Negative} 
is the best that can be attained by polynomials from the set ${\mathcal L}(n,\tau,h)$.}
\end{theorem}

\begin{proof} \textcolor{black}{We shall sketch the modifications in the proof of Theorem \ref{PULB}. We first consider when $\tau=2k-1$. Using the nodes $\alpha_0=-1$, the roots $\{\alpha_i\}_{i=1}^{k-1}$ of $P_{k-1}^{(1,1)}(t)$ and $\alpha_k=1$, we can define the nodes $\rho_i$ so that the quadrature
\begin{equation} \label{GJQ_1}
\gamma_n \int_{-1}^1 f(t) (1-t^2)^{\frac{n-3}{2}} \, dt=\sum_{i=0}^k \rho_i f(\alpha_i) 
\end{equation}
be exact for polynomials of degree $k$. Using the representation of a generic polynomial $f(t)$ of degree  $2k-1$ 
\[f(t)=(t-1)(t+1)P_{k-1}^{(1,1)}(t)q(t)+p(t),\]
where $q\in \Pi_{k-2}, p\in \Pi_k$, we show similarly the quadrature rule is exact on $\Pi_{2k-1}$. The interpolating polynomial $H_{2k-1}(t;h)$ interpolates $h$ at all $(k+1)$ nodes and interpolates $h^\prime$ at $\{\alpha_i\}_{i=1}^{k-1}$. The error formula \
\begin{equation}\label{HermiteError_Negative}
h(t)-  H_{2k-1}(t;h) = \frac{h^{(2k)}(\xi)}{(2k)!}(t^2-1)(t-\alpha_1)^2 \dots (t-\alpha_{k-1})^2\geq 0 \end{equation}
shows $H_{2k-1}\in {\mathcal L}(n,\tau,h)$. The optimality of $H_{2k-1}$ for the linear program \eqref{LP_ULB} is similarly derived.
}

\textcolor{black}{When $\epsilon=1$, we utilize as Hermite interpolating nodes the roots $\{\alpha_i \}_{i=1}^k$ of $P_k^{(1,0)}(t)$ and add $\alpha_{k+1}=1$ as Lagrange interpolation node. With the weights analogous to \eqref{RhoWeights} the quadrature rule
\begin{equation} \label{QR_1_epsilon}
\gamma_n \int_{-1}^1 f(t) (1-t^2)^{\frac{n-3}{2}} \, dt=\sum_{i=1}^{k+1} \rho_i f(\alpha_i) \nonumber
\end{equation}
is shown similarly to be exact on $\Pi_{2k}$. The error formula now takes the form
\begin{equation}\label{HermiteError1_Negative} 
h(t)-  H_{2k}(t;h) = \frac{h^{(2k+1)}(\xi)}{(2k+1)!}(t-1)(t-\alpha_1)^2 \dots (t-\alpha_k)^2\geq 0,\quad 
\end{equation}
which shows $H_{2k}\in {\mathcal L}(n,\tau,h)$ and the proof is completed as in the odd case. }
\end{proof}

Utilizing \eqref{PULB-PUUB} we have the following max-min polarization universal lower bound (PULB).  The bound \eqref{PULB_N} below is a direct consequence of the first inequality in \eqref{PULB-PUUB} and Definition \ref{Tau_{n,N}}.

\begin{corollary}\label{PULB-N} Let 
$\tau_{n,N}=2k-1+\epsilon$, where $\epsilon\in \{0,1\}$, namely assume there exists a spherical $\tau_{n,N}$-design $C\subset \mathbb{S}^{n-1}$ of cardinality $|C|=N$. Then for  any potential $h$ continuous on $[-1,1]$, finite on $[-1,1)$ for which $h^{(2k+\epsilon)}$ does not change sign in $(-1,1)$,
we have that
\begin{equation}\label{PULB_N}
\mathcal{Q}^{(h)}(N) \geq N\sum_{i\in I} \rho_i h(\alpha_i), 
\end{equation}
with the summation index set $I$, nodes $\{\alpha_i\}$ and weights $\{ \rho_i \}$ as described in Theorems \ref{PULB} and \ref{PULB_Negative}.
\end{corollary}

We next examine the conditions for achieving the polarization bounds \eqref{PolarizationULB} and \eqref{PolarizationULB_Negative}. 
 
\begin{corollary}\label{attaining-cor-1} Let $h$ be a potential function that  \textcolor{black}{satisfies the conditions of Theorem \ref{PULB} or Theorem \ref{PULB_Negative}}.
If a spherical $\tau$-design $C \subset \mathbb{S}^{n-1}$, $\tau=2k-1+\epsilon$, $|C|=N$, 
attains the bound \eqref{PolarizationULB} or \eqref{PolarizationULB_Negative}, then 
there exists a point $x \in \mathbb{S}^{n-1}$ such that the set $T(x,C)$ of all inner products between $x$ and the points of $C$ coincides with the set 
 $\{\alpha_i\}_{i\in I}$, and the multiplicities of these inner products are $\{N\rho_i \}_{i\in I}$, respectively, where $I$ is the index set above. In particular, the numbers $N\rho_i$, $i\in I$,
are positive integers.
\end{corollary}

\begin{remark} During a January 2022 workshop at the Erwin Schr\"{o}dinger Institute in Vienna, S.  Borodachov introduced the notion of $m$-stiff configurations, namely spherical $(2m-1)$-designs, for which there is a point $x\in \mathbb{S}^{n-1}$ such that the cardinality of the set $T(x,C)$ is exactly $m$. Such $m$-stiff configurations attain the bound \eqref{PolarizationULB} ($k=m$ in our notations). 
\end{remark}

\begin{proof} (of Corollary \ref{attaining-cor-1}) Observe that should the potential $h$ satisfy $h^{(2k+\epsilon)}(\xi)>0$ for all $\xi\in(-1,1)$, the Hermite  interpolation error formulas \eqref{HermiteError} and \eqref{HermiteError1} would hold with strict inequality for all $t\not\in \{\alpha_i \}_{i\in I}$. Similarly, $h^{(2k+\epsilon)}(\xi)<0$ implies \eqref{HermiteError_Negative} and \eqref{HermiteError1_Negative} hold with strict inequality. Therefore, if equality in \eqref{PolarizationULB} or \eqref{PolarizationULB_Negative} occurs for some spherical $\tau$-design $C$, then there is an $x \in \mathbb{S}^{n-1}$ such that 
\[x\cdot y \in \{\alpha_i \}_{i\in I}\quad  {\rm for \ all} \quad\ y\in C.\]
Suppose there is a $j\geq 1$ such that for all  $y\in C$ we have $x\cdot y \not=\alpha_j$. Denote by $p$ any of the polynomials on the right-hand sides of \eqref{HermiteError}, \eqref{HermiteError1}, \eqref{HermiteError_Negative}, and \eqref{HermiteError1_Negative}. Note that $p\in \Pi_{\tau+1}$, 
$p(x\cdot y)=0$ for all $y\in C$, and $0\not\equiv p(t)\geq 0$ on $[-1,1]$, so we have that
\[ 0=\sum_{y\in C} p(x\cdot y )=Np_0=N \gamma_n \int_{-1}^1  p(t)(1-t^2)^{(n-3)/2}\, dt\not=0.\]
Thus, we conclude that equality in \eqref{PolarizationULB} or \eqref{PolarizationULB_Negative} yields the existence of $x\in \mathbb{S}^{n-1}$ such that the collection $T(x,C)$ coincides with the quadrature nodes  $\{\alpha_i\}_{i\in I}$.

To show that the frequencies of these inner products are exactly $\{ N\rho_i \}_{i\in I}$ and are positive integers, we simply substitute in \eqref{defin_f} the Lagrange basic polynomials $\ell_i (t)$ associated with the quadrature nodes. This completes the proof.
\end{proof}
We now present two notable configurations attaining the bound \eqref{PolarizationULB}. Moreover, both configurations also attain the Fazekas-Levenshtein bound \eqref{FL_bound}, a fact that seems not to have been noticed before.

\begin{example} \label{Cube}
Let $C_0$ denote the cube inscribed in $\mathbb{S}^2$. Without loss of generality orient two of its sides to be horizontal with the vertices being $\{ (\pm1/\sqrt{3},\pm1/\sqrt{3},\pm1/\sqrt{3})\}$. The cube is a spherical $3$-design. This and the Delsarte-Goethals-Seidel bound for $4$-designs imply that $\tau_{3,8}=3$. The Gegenbauer polynomial $P_2^{(3)}(t)=(3t^2-1)/2$ has zeros $\alpha_1= -\frac{1}{\sqrt{3}}$ and  $\alpha_2=\frac{1}{\sqrt{3}}$. The corresponding weights are found to be $\rho_1=\rho_2=1/2$ (note that $8\rho_1=8\rho_2=4$), so the bound \eqref{PolarizationULB} takes the form
\begin{equation} \label{Cube_18}
\mathcal{Q}^{(h)}(3,3,8)\geq 4h\left(-\frac{1}{\sqrt{3}}\right)+4h\left(\frac{1}{\sqrt{3}}\right). \end{equation}
Denoting the North pole by $\widetilde{x}=(0,0,1)$ we have $T(\widetilde{x},C_0)=\{\alpha_1,\alpha_2\}$, both inner products have multiplicity 4,
so the conditions of Corollary \ref{attaining-cor-1} hold true and we see that $U_h(\widetilde{x},C_0)$ equals the right-hand side of \eqref{Cube_18}. 
Therefore, the cube is a universal optimal code with respect to  \eqref{PolarizationULB} in the sense that the nodes and weights do not depend on the potential $h$.
Moreover, Corollary \ref{PULB-N} implies that
\[ \mathcal{Q}^{(h)}(8)\geq 4h\left(-\frac{1}{\sqrt{3}}\right)+4h\left(\frac{1}{\sqrt{3}}\right)=U_h(\widetilde{x},C_0)=\mathcal{Q}^{(h)}(3,3,8). \] 

Note that by symmetry all intersections of the coordinate axes with the unit sphere will constitute points of minima of the discrete potential $U_h (x,C_0)$.
The graphs of the Newton potential ${\rm v}(t)=1/\sqrt{2(1-t)}$ and the Gaussian potential ${\rm w}(t)=e^{2(t-1)}$ and their corresponding Hermite interpolants $H_{3}(t;{\rm v})$ and $H_{3}(t;{\rm w})$ at $\alpha_1, \alpha_2$, respectively, are shown in Figure \ref{fig:cube} along with the touching nodes. \begin{figure}[htbp]
\centering
\includegraphics[width=3.5 in]{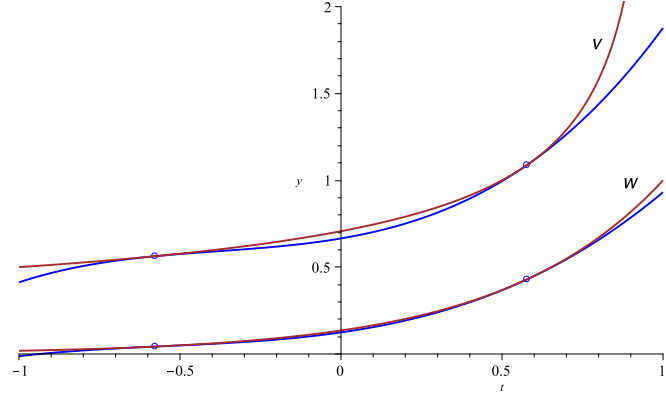}
\caption{$n=3$, $N=8$, $\tau=\tau_{3,8}=3$. Graphs of potentials $h$, Hermite interpolants $H_3(t;h)$, and nodes $\alpha_i$, $i=1,2$ for the Newton potential ${\rm v}(t)=1/\sqrt{2(1-t)}$, (PULB $=6.6027$)
and the Gauss potential ${\rm w}(t)=e^{2(t-1)}$, (PULB $=1.8883$).}
\label{fig:cube}
\end{figure}

To illustrate Theorem \ref{PULB_Negative} suppose $h^{(4)}(t)\leq 0$ on $(-1,1)$. The cube is a 3 design, so the nodes are $\alpha_0=-1$, $\alpha_1=0$ and $\alpha_2=1$. The bound \eqref{PolarizationULB_Negative} is computed to be
\[ \mathcal{Q}^{(h)}(8)\geq \mathcal{Q}^{(h)}(3,3,8)\geq 8\left(\frac{1}{6}h(-1)+\frac{2}{3}h(0)+\frac{1}{6}h(1)\right). \] 
\hfill $\Box$
\end{example}

\begin{example} \label{Lower24cell}
Assume $C\subset \mathbb{S}^3$ and $|C|=24$. There exists a collection of $24$-point $5$-designs on $\mathbb{S}^3$ which 
is a three-parameter family (see \cite{CCEK}). This and the Delsarte-Goethals-Seidel bound for 6-designs implies that $\tau_{4,24}=5$.
The Gegenbauer polynomial $P_3^{(4)}(t)=t(2t^2-1)$ has zeros $\{ -\frac{1}{\sqrt{2}},0,\frac{1}{\sqrt{2}}\}$. The graphs of
the Newton potential ${\rm v}(t)=1/(2(1-t))$ and the Gaussian potential ${\rm w}(t)=e^{2(t-1)}$ and the corresponding Hermite interpolants
(i.e. respective optimal interpolation polynomials) 
$H_{5}(t;{\rm v})$ and $H_{5}(t;{\rm w})$ are shown in Figure \ref{fig:1} along with the touching nodes.

\begin{figure}[htbp]
\centering
\includegraphics[width=3.5 in]{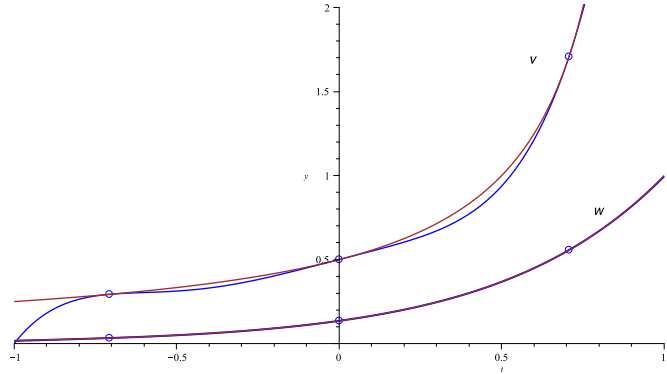}
\caption{$n=4$, $N=24$, $\tau=\tau_{4,24}=5$.  Graphs of potentials $h$, Hermite interpolants $H_5(t;h)$, and nodes $\alpha_i$, $i=1,2,3$ for the Newton potential ${\rm v}(t)=\frac{1}{2(1-t)}$ (PULB $=18$)
and the Gauss potential ${\rm w}(t)=e^{2(t-1)}$ (PULB $=5.1614$).}
\label{fig:1}
\end{figure}

The corresponding weights are $\{\frac{1}{4},\frac{1}{2},\frac{1}{4} \}$, so the bound \eqref{PolarizationULB} takes the form
\[ \mathcal{Q}^{(h)}(4,5,24)\geq 6h\left(-\frac{1}{\sqrt{2}}\right)+12h(0)+6h\left(\frac{1}{\sqrt{2}}\right). \] 
Note that $24\rho_1=24\rho_3=6$ and $24\rho_2=12$ are positive integers. 

Of particular interest is the $5$-design kissing number configuration $D_4$, or the so-called $24$-cell (see \cite{CCEK, M}). The coordinates of its $24$ vertices are formed by the permutations of $(\pm \frac{1}{\sqrt{2}}, \pm \frac{1}{\sqrt{2}},0,0)$. For any of the points $\widetilde{x}$ whose coordinates are a permutation of $(\pm 1,0,0,0)$, that is the intersections with the coordinate axes, we calculate that
\[ U_h(\widetilde{x},D_4)= 6h\left(-\frac{1}{\sqrt{2}}\right)+12h(0)+6h\left(\frac{1}{\sqrt{2}}\right),\]
thereby deriving that these points are minima of the polarization potential. Therefore, the 24-cell is universally optimal with respect to the bound  \eqref{PolarizationULB}.

We can determine all points $\widetilde{x}$ where the minimun is attained as follows. Suppose that a point $\widetilde{x}=(x_1,x_2,x_3,x_4)\in \mathbb{S}^3$ is such a minimum. From Corollary \ref{attaining-cor-1} we derive that the inner products from $x$ to the points in $D_4$ will have to be in $A:=\{ -\frac{1}{\sqrt{2}},0,\frac{1}{\sqrt{2}}\}$ with frequencies $\rho:=\{ 6, 12, 6\}$, respectively. If there is a coordinate $x_i=0$, then there has to be a coordinate $x_j\not= 0$ and by selecting the point from $D_4$ with $1/\sqrt{2}$ at $i$-th and $j$-th place we obtain that the inner product between $x$ and this point is $x_j /\sqrt{2} \in A$, which implies $x_j=\pm 1$, minima that is already accounted for. Therefore, we may assume all $x_i\not= 0$. If there are two indexes $i\not= j$, such that $0<|x_i|<|x_j|$, then by selecting appropriate point from $D_4$ we shall obtain an inner product $0<(|x_j|-|x_i|)/\sqrt{2}<1/\sqrt{2}$, which is a contradiction. Therefore, all minima that are either permutations of $(\pm 1,0,0,0)$, or of the type $\left(\pm \frac{1}{2},\pm \frac{1}{2},\pm \frac{1}{2},\pm \frac{1}{2}\right)$. 

\textcolor{black}{ As an application of Theorem \ref{PULB_Negative}, when $h^{(6)}(t)\leq 0$ on $(-1,1)$ we have the PULB bound
\[ \mathcal{Q}^{(h)}(4,5,24)\geq 24\left[ \frac{1}{20}h(-1) + \frac{9}{20}h\left(-\frac{1}{\sqrt{6}}\right)+\frac{9}{20}h\left(\frac{1}{\sqrt{6}}\right)+\frac{1}{20}h(1)\right]. \] 
}\hfill $\Box$\end{example}

\textcolor{black}{ Observe, that since $N\rho_i$ are not integers, neither the cube, nor the $24$-cell achieve the bound \eqref{PolarizationULB_Negative} in general. Our next example shows a configuration that attains the bound \eqref{PolarizationULB_Negative}. We shall pursue in more detail the codes that are optimal for Theorem \ref{PULB_Negative} in a subsequent manuscript.}
\textcolor{black}{ 
\begin{proposition} \label{Negative_Cross}Let $h$ be a potential such that $h^{(4)}(t)\leq 0$ on $(-1,1)$. Then the cross-polytope $C_{2n}=\{\pm e_i,i=1,\dots,n\}$ is a universal optimal code with respect to \eqref{PolarizationULB_Negative}, in the sense that the nodes and weights do not depend on the potential $h$.
\end{proposition}
\begin{proof} Since the cross-polytope is a $3$ design, the corresponding nodes for $3$ designs are $\alpha_0=-1$, $\alpha_1=0$, and $\alpha_2=1$. Computing the weights we derive from \eqref{PolarizationULB_Negative} that 
\[ Q^{(h)}(n,3,2n)\geq h(-1)+(2n-2)h(0)+h(1).\]
This clearly is satisfied for all vertices of the cross-polytope.
\end{proof}}
\section{Universal upper bounds for polarization}\label{UpperBound}

We next utilize \eqref{defin_f} to derive upper bounds on the maximal polarization constant for spherical designs based on their covering radius. 

\begin{definition} \label{UpperClass} 
Let $\tau$ be a positive integer, $h$ be a potential function, and $s\in(-1,1]$. Denote by 
${\mathcal A}(n,s,\tau,h)$  the class of {\em upper admissible polynomials}  $f$ that satisfy 
\begin{itemize}
\item[(B1)] $\deg(f) \leq \tau$;
\item[(B2)] $f(t) \geq h(t)$ for every $t \in [-1,s]$.
\end{itemize}
\end{definition}

\begin{proposition} \label{prop42}
Let $h(t)$ be a potential function, $\tau$ be a positive integer, $C\subset \mathbb{S}^{n-1}$ be a spherical $\tau$-design of 
cardinality $N$, and $s=s_C \leq s_{\tau,N}$. Let $f \in {\mathcal A}(n,s,\tau, h)$ be an upper admissible polynomial. Then 
the following upper bound
\begin{equation}\label{upper-bound}
U_h(x,C) \leq  U_f(x,C) = f_0|C|
\nonumber
\end{equation}
 holds for every point $x \in \mathbb{S}^{n-1}$ such that $T(x,C) \subset [-1,s]$. 
Subsequently,
\begin{equation}\label{DY_UB}
\mathcal{Q}_h(C) \leq \min_{f\in {\mathcal A}(n,s,\tau,h)} \{ f_0|C| \}. 
\end{equation}
\end{proposition}

\begin{proof}
The condition (B1) implies that $U_f(x,C) = f_0|C|$. From (B2) we derive $U_h(x,C) \leq  U_f(x,C)$ for every $x \in \mathbb{S}^{n-1}$ such that
$T(x,C) \subset [-1,s]$.  Note that such points $x$ exist because of the assumption $s=s_C \leq s_{\tau,N}$. Now \eqref{DY_UB} 
follows from \eqref{max-min} and \eqref{upper-bound}. 
\end{proof}

The following corollary is an immediate consequence of \eqref{DY_UB}.
\begin{corollary} If the collection of spherical $\tau$-designs $C \subset \mathbb{S}^{n-1}$ of cardinality $|C|=N$ is non-empty, then
\begin{equation} \label{Polarization_UB1}
\mathcal{Q}^{(h)}(n,\tau,N) \leq \min_{f\in {\mathcal A}(n,s_{\tau,N},\tau,h)} \{ f_0N \}. 
\end{equation}
\end{corollary}

The case $s=1$ does not impose restrictions on the points $x$ in Proposition \ref{prop42} and we derive a upper bound for the quantity $\mathcal{R}_h(C)$, 
\begin{equation}\label{UpperB}
\mathcal{R}_h(C)\le f_0|C|,
\end{equation}
provided $f \in \mathcal{A}(n,1,\tau,h)$ and $h$ is finite and continuous at $1$.

The above considerations give rise to the linear program
\begin{align}\label{LP_UUB}
{\rm given}: & \quad n, \ \tau, \ s, \ N, \ h \nonumber \\
{\rm minimize:} & \quad f_0  \\
{\rm subject \ to:} &\quad  f \in {\mathcal A}(n, s,\tau,h)\nonumber
\end{align}

The presence of the parameter $s$ requires another approach (compared to Section 3). In analyzing the program \eqref{LP_UUB} 
we shall use the Gram-Schmidt orthogonalization process with respect to the signed measure 
\[d\nu^{s,1-\epsilon} (t):=\gamma_n (1+t)^{1-\epsilon}(s-t)(1-t^2)^{(n-3)/2}\, dt,\] 
where $\epsilon\in\{0,1\}$, so we recall the following definition (see \cite[Definition 3.4]{CK} and \cite[Definition 2.1]{BDHSS-DCC19}).

\begin{definition}\label{nu_pos}
A signed Borel measure $\nu$ on $\mathbb{R}$ for which all polynomials are integrable is called {\em positive definite up to degree $m$}  if for all real polynomials $p\not\equiv 0$ of degree at most $m$ we have $ \int p(t)^2\, d\nu(t) >0$.
\end{definition} 

The following lemma investigates the positive definiteness of the measure $\nu^{s,1-\epsilon} (t)$ in terms of $s$.

\begin{lemma} \label{SignedNu_s}  Let $k\geq 1$ and let $t_{k-1+\epsilon}^{0,1-\epsilon }\leq s\leq 1$, where $t_{k-1+\epsilon}^{0,1-\epsilon}$ is the largest root of the $(k-1+\epsilon)$-th (possibly adjacent) Gegenbauer polynomial $P_{k-1+\epsilon}^{0,1-\epsilon}(t)$. Then the signed measure $\nu^{s,1-\epsilon} (t)$ is positive definite up to degree $k-2+\epsilon$.
\end{lemma}

\begin{proof} Let $ q(t)\in \Pi_{k-2+\epsilon}$, $q(t) \not\equiv 0$. Using the Gauss-Jacobi quadratures \eqref{QR_epsilon} or \eqref{GJQ}  for $\epsilon=0$ and $\epsilon=1$, respectively, we obtain
\[ \gamma_n \int_{-1}^1 q(t)^2 (1+t)^{1-\epsilon} (s-t)(1-t^2)^{(n-3)/2}\, dt=\sum_{i=\epsilon}^{k-1+\epsilon} \rho_i q(\alpha_i)^2(s-\alpha_i)>0,\]
because $s\geq t_{k-1+\epsilon}^{0,1-\epsilon}=\alpha_{k-1+\epsilon}$. This proves the lemma.
\end{proof}

The positive definiteness of $\nu^{s,1-\epsilon}$ allows us to perform a corresponding Gram-Schmidt orthogonalization procedure and find orthogonal polynomials $\{ q_j^{s,1-\epsilon} (t)\}_{j=0}^{k-1+\epsilon}$ with respect to $\nu^{s,1-\epsilon}$ up to degree $k-1+\epsilon$, which is key in establishing the upper bound for the maximal polarization of designs, an analog of Theorem \ref{PULB}.

Let $\{ \beta_i \}_{i=1}^{k-1+\epsilon}$ be the $k-1+\epsilon$ distinct zeros of the polynomial $q_{k-1+\epsilon}^{s,1-\epsilon}(t)$, 
all in the interval $(-1,s)$. We also set $\beta_0=-1$ when $\epsilon=0$ and $\beta_{k+\epsilon}=s$ either for $\epsilon=0$ and 1, 
assuming the ordering 
\[ -1=\beta_0<\beta_{1}<\cdots<\beta_{k-1+\epsilon}<\beta_{k+\epsilon}=s. \]

\begin{theorem} \label{PUUB} Suppose the potential $h$ has non-negative derivative of order $2k+\epsilon$, $\epsilon\in \{0,1\}$. 
Assume there exists a spherical $\tau$-design $C$
of cardinality $N$ on $\mathbb{S}^{n-1}$, where $\tau=:2k-1+\epsilon$, and let $s=s_{\tau,N}$.
Then the following {\em Polarization-UUB} hold:
\begin{equation}\label{PolarizationUUB}
\mathcal{Q}^{(h)}(n,\tau,N) \leq N  \sum_{i=\epsilon}^{k+\epsilon} r_i h(\beta_i), 
\end{equation}
and
\begin{equation}\label{PolarizationUUB_2}
\mathcal{R}^{(h)}(n,\tau,N) \leq N \sum_{i=\epsilon}^{k+\epsilon} r_i h(\beta_i), 
\end{equation}
with nodes $\{\beta_i\}_{i=\epsilon}^{k+\epsilon}$ as explained above. The corresponding weights $\{ r_i \}_{i=\epsilon}^{k+\epsilon}$ given by 
\begin{equation}\label{weights_beta} 
r_i=\gamma_n\int_{-1}^1 \ell_i (t) (1-t^2)^{(n-3)/2}\, dt ,
\end{equation} 
where $\ell_i(t)$ are the Lagrange basic polynomials associated with the nodes $\{ \beta_i \}_{i=\epsilon}^{k+\epsilon}$, are positive and sum to $1$.
\end{theorem}

\begin{proof} We shall prove \eqref{PolarizationUUB}, the bound \eqref{PolarizationUUB_2} being mere application of \eqref{UpperB}.

We first note the inequalities $t_{k-1+\epsilon}^{0,1-\epsilon}\leq t_k^{0,\epsilon}$, following from \cite[Lemma 3.3]{Lev92} for $\epsilon=1$ 
and \cite[Lemma 3.4]{Lev92} for $\epsilon=0$, respectively. Since $t_k^{0,\epsilon}\leq \ell_{\tau,N}\leq s_{\tau,N} = s$ from \eqref{FL_bound}, 
Lemma \ref{SignedNu_s} can be applied. Then Definition \ref{nu_pos} and Lemma \ref{SignedNu_s} imply that 
\[ \langle f,g \rangle_{\nu^{s,1-\epsilon}} := \int_{-1}^1 f(t) g(t) \, d \nu^{s,1-\epsilon} (t) \]
defines an inner product on $\Pi_{k-2+\epsilon}$. Therefore, 
given a basis $\{p_j\}_{j=0}^{k-1+\epsilon}$ of $\Pi_{k-1+\epsilon}$ the Gram-Schmidt orthogonalization procedure
\begin{equation} \label{sOrtho}
q_0^{s,1-\epsilon}(t)=p_0(t),\quad q_j^{s,1-\epsilon}(t)=p_j-\sum_{i=0}^{j-1}\left( \frac{\langle p_j,q_i^{s,1-\epsilon}\rangle_{\nu^{s,1-\epsilon}}}{\langle q_i^{s,1-\epsilon},q_i^{s,1-\epsilon} \rangle_{\nu^{s,1-\epsilon}}}\right) q_i^{s,1-\epsilon}(t),\quad j=1,\dots,k-1,
\nonumber
\end{equation}
yields an orthogonal basis with respect to $\nu^{s,1-\epsilon}$. Suppose the weights $r_i$, $i=\epsilon,\dots,k+\epsilon$ are selected according to \eqref{weights_beta}. This guarantees that the quadrature formula
\begin{equation} \label{GJs}
\gamma_n \int_{-1}^1 f(t) (1-t^2)^{\frac{n-3}{2}} \, dt=\sum_{i=\epsilon}^{k+\epsilon} r_i f(\beta_i) 
\nonumber
\end{equation}
is exact on $\Pi_k$. As in the proof of Theorem \ref{PULB} we can show the exactness of \eqref{GJs} on $\Pi_{2k-1+\epsilon}$.

%

The positivity of the weights follows similarly as in Theorem \ref{PULB} with one exception. When $\epsilon=1$, we substitute in \eqref{GJs} the polynomials $\ell_i^2(t)$, which yields that $r_i>0$. When $\epsilon=0$ and $i\not= 0$, we use $\ell_i^2(t)/(1+t)\in \Pi_{2k-1}$. Finally, for $\epsilon=0$ and $i=0$, we have from \eqref{GJs} and \eqref{GJQ}
\begin{eqnarray*}
r_0 \left[ q_{k-1}^{s,1} (-1) \right]^2 (s+1) &=& \gamma_n \int_{-1}^1 \left[ q_{k-1}^{s,1} (t) \right]^2 (s-t) (1-t^2)^{\frac{n-3}{2}}\, dt \\
&=& \sum_{i=1}^k \rho_i  \left[ q_{k-1}^{s,1} (\alpha_i) \right]^2 (s-\alpha_i)>0. 
\end{eqnarray*}
 Recall that the nodes $\{\alpha_i\}$ are the roots of $P_k^{(0,0)} (t)$ and $s\geq t_k^{0,0}$, which implies that $r_0>0$. That the weights sum to $1$ is analogous to the argument in Theorem \ref{PULB}.
%

We next define the Hermite interpolation polynomial $H_{2k-1+\epsilon}(t;h,s)$ such that $ H_{2k-1+\epsilon}(\beta_i;h,s)=h(\beta_i)$ for $i=\epsilon,\dots,k+\epsilon$ and $H_{2k-1+\epsilon}^\prime(\beta_i;h,s)=h^\prime (\beta_i)$, $i=1,\dots,k-1+\epsilon$. We again utilize the Hermite interpolation error formula (see \cite[Lemma 2.1]{CK}) and derive that for any $t\in [-1,s]$, there exists a $\xi\in (-1,s)$ such that
\begin{equation}\label{HermiteError2}
h(t)-  H_{2k-1+\epsilon}(t;h,s) = \frac{h^{(2k+\epsilon)}(\xi)}{(2k+\epsilon)!}(1+t)^{1-\epsilon}(t-\beta_1)^2 \dots (t-\beta_{k-1+\epsilon})^2(t-s)\leq 0.\end{equation}

Therefore, $f:=H_{2k-1+\epsilon} \in {\mathcal A}(n,s,\tau,h)$. Since $s=s_{\tau,N}$, for any spherical $\tau$-design $C$ we have that $s_C\leq s$.
This means that for some $x^*\in \mathbb{S}^{n-1}$, $( x^*\cdot y )\leq s$ for all $y\in C$, and hence
\[U_h(C)\leq U_h(x^*,C)\leq U_f(x^*,C) =f_0N=N \sum_{i=\epsilon}^{k+\epsilon} r_i h(\beta_i)\]
(here $f_0$ is the constant coefficient in the Gegenbauer expansion of $H_{2k-1+\epsilon}$).
Taking the supremum over all such $\tau$-designs of cardinality $N$ we obtain \eqref{PolarizationUUB}.

To see that the bound is optimal among polynomials in ${\mathcal A}(n,s,\tau,h)$, suppose that $f$ is an arbitrary polynomial in the class. We then have 
\[ Nf_0=N\gamma_n \int_{-1}^1 f(t) (1-t^2)^{(n-3)/2}\, dt=N \sum_{i=\epsilon}^{k+\epsilon} r_i f(\beta_i)\geq N \sum_{i=\epsilon}^{k+\epsilon} r_i h(\beta_i).\]
Therefore, $H_{2k-1+\epsilon}(t;h,s)$ is the unique polynomial attaining the Linear Programming bound \eqref{PolarizationUUB}. 
\end{proof}

For potentials $h$ that are continuous at $t=1$, we can select $s=1$ in Theorem \ref{PUUB} and utilize \eqref{PULB-PUUB} to derive the following min-max polarization universal upper bound (PUUB).  

\begin{corollary}\label{PUUB-N} Let $h$ be a \textcolor{black}{potential with non-negative derivative of order $2k+\epsilon$, $\epsilon\in \{0,1\}$}, that is finite and continuous at $t=1$. If there exists a spherical design $C \subset \mathbb{S}^{n-1}$ of cardinality $|C|=N$ and strength $\tau_{n,N}=2k-1+\epsilon$, then 
\begin{equation}\label{PUUB_N}
\mathcal{R}^{(h)}(N) \leq N\cdot \sum_{i=\epsilon}^{k+\epsilon} r_i h(\beta_i), 
\end{equation}
where the nodes $\{\beta_i\}_{i=1}^{k-1+\epsilon}$ are the roots\footnote{Note that these are exactly the 
Levenshtein parameters as defined in \cite[Theorem 5.39]{Lev} (see also \cite[Theorems 4.1 and 4.2]{Lev92})}
 of the adjacent Gegenbauer polynomials $P_{k-1+\epsilon}^{1,1-\epsilon}$, $\beta_0=-1$, $\beta_{k+\epsilon}=1$ and the weights $\{ r_i \}_{i=\epsilon}^{k+\epsilon}$ are defined in \eqref{weights_beta}.\end{corollary}
\begin{remark} Note that as $s\to 1^-$, the polynomials $q_{k-1+\epsilon}^{s,1-\epsilon}$ from Theorem \ref{PUUB} approach the polynomials $P_{k-1+\epsilon}^{1,1-\epsilon}$ from Corollary \ref{PUUB-N}.
\end{remark}

\begin{proof} Using that for such a design $C$ we have
\[U_h(x,C)\leq U_f(x,C) =f_0N=N \sum_{i=\epsilon}^{k+\epsilon} r_i h(\beta_i)\]
(now $f_0$ is the constant coefficient in the Gegenbauer expansion of $f=H_{2k-1+\epsilon}(t;h,1)$),
we can take maximum over $x$ and derive 
\[ \mathcal{R}_h (C) \leq N \sum_{i=\epsilon}^{k+\epsilon} r_i h(\beta_i),\]
which implies the Corollary.
\end{proof}

\begin{corollary}\label{attaining-cor} 
If a spherical $\tau$-design $C \subset \mathbb{S}^{n-1}$, $\tau=2k-1+\epsilon$, $\epsilon \in \{0,1\}$, $|C|=N$, attains the bound \eqref{PolarizationUUB} for \textcolor{black}{a potential with positive derivative of order $2k+\epsilon$}, then 
there exist a point $x \in \mathbb{S}^{n-1}$ such that the set $T(x,C)$ of all inner products between $x$ and the points of $C$ coincides with the set 
 $\{\beta_i\}_{i=\epsilon}^{k+\epsilon}$ and the multiplicities of these inner products are $\{Nr_i \}_{i=\epsilon}^{k+\epsilon}$, respectively. In particular, the numbers $Nr_i$, $i=\epsilon,\ldots,k+\epsilon$, are positive integers. 
\end{corollary}
\begin{proof} The proof is analogous to the proof of Corollary \ref{attaining-cor-1}.
\end{proof}

The polarization bounds from Theorem \ref{PULB} and Theorem \ref{PUUB} provide a strip, where the min-max and max-min polarizations of 
spherical designs could live, namely
\[ N\cdot \sum_{i=1-\epsilon}^{k} \rho_i h(\alpha_i) \leq \mathcal{Q}^{(h)}(n,\tau,N),\mathcal{R}^{(h)}(n,\tau,N) \leq 
N \cdot \sum_{i=\epsilon}^{k+\epsilon} r_i h(\beta_i). \]
This will be discussed in detail in a future work but we give here some illustrations, 
providing polarization universal upper bound analogs to Examples \ref{Cube} and \ref{Lower24cell} about the cube on $\mathbb{S}^2$ and the $24$-cell $D_4$ on $\mathbb{S}^3$.

\begin{figure}[htbp]
\centering
\includegraphics[width=3.5 in]{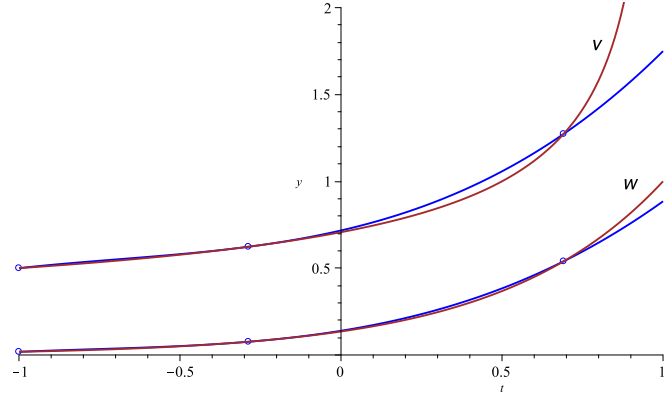}
\caption{$n=3$, $N=8$, $\tau=\tau_{3,8}=3$, $s=0.691$.  Graphs of potentials $h$, Hermite interpolants $H_3(t;h,0.691)$, and nodes $\beta_i$, $i=0,1,2$, 
for the Newton potential ${\rm v}(t)=1/\sqrt{2(1-t)}$ (PUUB $=6.8239$) and the Gauss potential ${\rm w}(t)=e^{2(t-1)}$ (PUUB $=1.9472$).}
\label{fig:8}
\end{figure}
\begin{example}\label{UpperCube}  As noted immediately before Example \ref{Cube}, the cube attains the Fazekas-Levenshtein 
bound \eqref{FL_bound}, so $s_{3,8}\geq 1/\sqrt{3}=\ell_{3,8}$. Let us fix $1/\sqrt{3}<s\leq 1$. In this case $\epsilon=0$, so the measure 
$d\nu^{s,1}(t)=1/2(1+t)(s-t)dt$ is positive definite up to degree 1. The corresponding orthogonal polynomial is $q_1^{s,1}=t-\beta_1$, where $\beta_1=-(1-s)/(3s-1)$. The other two quadrature nodes are $\beta_0=-1$ and $\beta_2=s$. When\footnote{This value is an upper bound for $s_{3,8}$ and is  computed with a method from \cite{St}.}  $s=0.691$, we compute PUUB for Coulomb potential as $6.8239$ and for Gauss potential as $1.9472$, respectively (see Figure \ref{fig:8}).
Selecting $s=1$ and Gauss potential ${\rm w}$, we obtain $\mathcal{R}^{({\rm w})}(8)\leq 2.0795$.

 \hfill $\Box$
\end{example}

\begin{example}\label{Upper24cell} Next we shall consider the PUUB for spherical $5$-designs of $24$ points on $\mathbb{S}^3$, which includes the three-parameter family referenced in Example \ref{Lower24cell} and \cite{CCEK}. From the Fazekas-Levenshtein bound \eqref{FL_bound} 
we have that $\ell_{5,24} = s_{D_4}=1/\sqrt{2}$. Determining the exact value of $s_{5,24}$ is a challenging problem. From \cite[Table 5.3]{St} we have an upper bound $s_{5,24}\leq s^*\approx 0.793867$. Below we fix $s>1/\sqrt{2}$ and determine the monic orthogonal polynomials up to degree two with respect to the measure $d\nu^{s,1}:=(2/\pi)(1+t)(s-t)\sqrt{1-t^2}\, dt$ utilizing the Gram-Schmidt orthogonalization procedure \eqref{sOrtho} applied to the standard basis $\{1,t,t^2\}$. We derive that 
\[ q_0^{s,1}(t)=1,\quad q_1^{s,1}(t)=t+\frac{1-s}{4s-1}, \quad q_2^{s,1}(t)=t^2+\frac{2s(1-s)}{6s^2-2s-1}t-\frac{2s^2-1}{2(6s^2-2s-1)}.
\]
The quadrature nodes are respectively
\begin{equation}\label{24s^*} \beta_0=-1,\quad \beta_{1,2}= \frac{2s(s-1) \mp  \sqrt{28s^4-16s^3-12s^2+4s+2}}{2(6s^2-2s-1)}, \quad \beta_3=s .
\end{equation}
For $s=s^*\approx 0.794$ we obtain upper bounds on the polarization constant for Newton and Gauss potential of $19.0819$ and $5.1675$, respectively (see Figure \ref{fig:2}). Recall our lower bounds from Example \ref{Lower24cell} of $18$ and $5.1614$, respectively. Together these provide the following corollary. 
\begin{corollary}\label{5des} The max-min and min-max polarization quantities of $5$-designs on $24$ points on $\mathbb{S}^3$ satisfy the following bounds
\[ 6h\left(-\frac{1}{\sqrt{2}}\right)+12h(0)+6h\left(\frac{1}{\sqrt{2}}\right) \leq \mathcal{Q}(4,5,24)\leq \mathcal{R}(4,5,24)\leq 
\sum_{i=0}^3 r_ih(\beta_i),\]
where the nodes $\beta_i$ are determined from \eqref{24s^*} with $s^*=0.794$ and $r_i$ are the weights from \eqref{weights_beta}. In particular, for the Newton potential ${\rm v}(t)=1/(2(1-t))$ we have the min-max and the max-min polarization bounds
\[ 18\leq \mathcal{Q}(4,5,24)\leq \mathcal{R}(4,5,24)\leq 19.0819.\]
\end{corollary}

\begin{figure}[htbp]
\centering
\includegraphics[width=3.5 in]{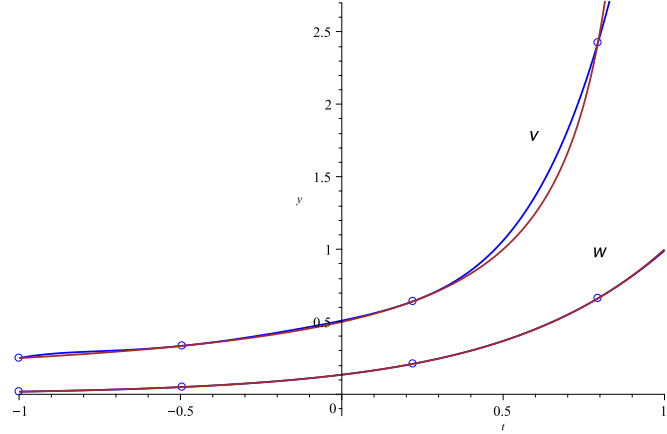}
\caption{$n=4$, $N=24$, $\tau=\tau_{4,24}=5$, $s=0.794$.  Graphs of potentials $h$, Hermite interpolants $H_5(t;h,0.794)$, and nodes $\beta_i$, $i=0,1,2,3$ for the Newton potential ${\rm v}(t)=\frac{1}{2(1-t)}$ (PUUB $=19.0819$) and the Gauss potential ${\rm w}(t)=e^{2(t-1)}$ (PUUB $=5.1675$).}
\label{fig:2}
\end{figure}
For continuous potentials we can apply the above framework for $s=1$ and illustrate Corollary \ref{PUUB-N}  (see Figure \ref{fig:3}). The corresponding nodes and weights are 
\[\beta_{0,3}=\mp 1,\quad \beta_{1,2}=\mp \frac{1}{\sqrt{6}},\quad r_{0,3}=\frac{1}{20}, \quad r_{1,2}=\frac{9}{20}.\]

\begin{figure}[htbp]
\centering
\includegraphics[width=3.5 in]{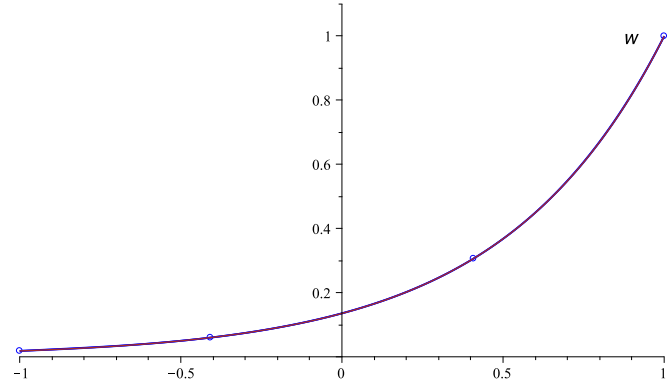}
\caption{$n=4$, $N=24$, $\tau=\tau_{4,24}=5$, $s=1$. Graphs of Gauss potential ${\rm w}(t)=e^{2(t-1)}$, Hermite interpolant $H_5(t;{\rm w},1)$, and nodes $\beta_i$, $i=0,1,2,3$ (PUUB $=5.17499$).}
\label{fig:3}
\end{figure}

In the case of Gauss potential ${\rm w}(t)=e^{2(t-1)}$ we compute the upper bound as $\mathcal{R}^{({\rm w})}(24)\leq 5.17499$.
\end{example}\hfill $\Box$

\section{Min-max polarization universal optimality of the $600$-cell}\label{Sec600}
 
 We now proceed with the min-max polarization problem for spherical codes of cardinality $N=120$ on $S^3$. As we said in the introduction Borodachov \cite{B} showed that strong sharp codes and sharp antipodal codes are universally optimal for the min-max polarization problem stated in \eqref{min-max}. The goal of this section is to show that the $600$-cell, \textcolor{black}{the unique $11$-design with $120$ points on $\mathbb{S}^3$ (see \cite{A2, Boy})}, satisfies the same universal optimality, which is in line with universal optimality of the sharp codes and the $600$-cell in terms of energy minimization for absolutely monotone potentials. 

We shall use the fact that the first nineteen moments of the $600$-cell except $M_{12}$ are all zero (see \eqref{Moments} for the definition of moments of a code). Recall that the eight distinct inner products of points in the $600$-cell are 
\[ B:=\left\{ -1, -\frac{1+\sqrt{5}}{4},-\frac{1}{2},\frac{1-\sqrt{5}}{4},0,\frac{\sqrt{5}-1}{4},\frac{1}{2}, \frac{1+\sqrt{5}}{4} \right\} =\{b_i\}_{i=1}^8.\]

Define $b_9:=1$. Given a potential $h(t)$ we denote with $g(t;h)$ the uniquely defined polynomial of degree $15$ that interpolates $h$ in $b_i$, $i=1,\dots,9$ and interpolates $h^\prime (t)$ in $b_i$, $i=2,\dots,8$. If $h(t)$ is \textcolor{black}{absolutely monotone of high enough order}, then the Hermite interpolation error formula implies $h(t)\leq g(t;h)$. However, when we expand $g(t;h)$ with respect to the Gegenbauer polynomials as a basis, its coefficient $g_{12}$ corresponding to $P_{12}^{(4)}(t)$ is not, in general, zero. To overcome this, we shall interpolate in the subspace
\[ {\mathcal P}:={\rm span} \{P_{0}^{(4)}(t), P_{1}^{(4)}(t),\dots,P_{11}^{(4)}(t), P_{13}^{(4)}(t), \dots ,P_{16}^{(4)}(t) \}=\Pi_{16}\cap \{ P_{12}^{(4)} \}^\perp.\]
Such an interpolation technique on a subspace was utilized by the authors in \cite{BDHSSMathComp} to derive an alternative proof of the universally optimal energy property of the $600$-cell. 

\begin{theorem}\label{600-cell} Let $h$ be a \textcolor{black}{continuous on $[-1,1]$ potential that is \textcolor{black}{absolutely monotone} of order $16$} and let $C_0 \subset \mathbb{S}^3$ denote the $600$-cell. The maximum of  $U_h(x,C_0)$ is attained at $y$ if and only if $y\in C_0$. Moreover, we have that for any other (i.e., non-isometric with $C_0$) code  $C\subset \mathbb{S}^3$ of cardinality $|C|=120$ the strict inequality $\mathcal{R}_h (C) > \mathcal{R}_h (C_0)$ holds. Consequently, 
\begin{equation}\label{600cellBound} \mathcal{R}^{(h)}(120)=h(b_1)+12\sum_{j=1}^4h(b_{2j})+20\sum_{i=0}^1 h(b_{4i+3})+30h(b_5) + h(b_9).\end{equation}
\end{theorem}
\begin{remark}\label{600Cell Assumptions}
\textcolor{black}{Even though stated for absolutely monotone potentials, an inspection of the proofs of the universal optimality of the $600$-cell presented in \cite{CK} (see also \cite[Theorems 5.1 and 5.2]{BDHSSMathComp}) reveals that the optimality actually holds for all absolutely monotone potentials of order $16$, which explains the assumptions of the theorem.} 
\end{remark}

\begin{proof} (of Theorem \ref{600-cell})
Consider the multi-set 
\[T=\{ t_j \}_{j=1}^{16}:=\{b_1,b_2,b_2,b_3,b_3,\dots,b_8,b_8,b_9\}. \]
Then we say that the interpolation polynomial $g(t)$ described above interpolates the potential function $h(t)$ at the points of the multiset $T$. Denote the partial products 
\[ g_j(t):=(t-t_1)\cdots(t-t_j), \ j=1,\dots, 16,\ g_0(t):=1.\]
Then the Newton interpolating formula yields 
\begin{equation}\label{NewtonFormula600cell} g(t;h)=\sum_{j=0}^{15} h[t_1,\dots,t_{j+1}]g_j(t) ,\end{equation}
where $h[t_1,\dots,t_{j+1}]$ denotes the divided difference of $h$ in the listed nodes.
Note that all of the divided differences of $h(t)$ in \eqref{NewtonFormula600cell} are nonnegative because of the absolute monotonicity of order $16$ of $h(t)$. From the Hermite interpolation error formula we have that (recall that $b_9=1$)
\begin{equation}\label{HermiteError600cell}
h(t)-  g(t;h) = \frac{h^{(16)}(\xi)}{16!}(t-b_1)(t-b_2)^2 \dots (t-b_8)^2(t-b_9)\leq 0, \quad t\in [-1,1].
\nonumber
\end{equation}
Therefore $h(t)\leq g(t;h)$. 

We shall next establish that the $12$-th Gegenbauer coefficient of $g(t;h)$ is nonnegative. Clearly, the $12$-th Gegenbauer coefficients of the partial products $g_j (t)$, $j=0,1,\dots,11$ are zero, i.e. $(g_j)_{12}=0$ for $j=0,1,\dots,11$. We can compute directly that 
\[ (g_{12})_{12}=\frac{1}{2^{12}},\quad (g_{13})_{12}=\frac{3+\sqrt{5}}{2^{13}},\quad (g_{14})_{12}=\frac{15+3\sqrt{5}}{2^{15}},\quad (g_{15})_{12}=\frac{3}{2^{14}} .\]   
That $(g(t;h))_{12}\geq 0$ follows now easily from \eqref{NewtonFormula600cell}.

We shall also need that $(g_{16})_{12}=2^{-14}>0$. With this in mind consider the polynomial of degree $16$
\[H(t;h):=g(t;h)-\frac{(g(t;h))_{12}}{(g_{16})_{12}} g_{16}(t) = \sum_{i=0}^{16} H_{i}P_i^{(4)}(t).\]
Clearly, $H(t;h)\geq h(t)$ and its Gegenbauer coefficient $H_{12}=0$. Moreover, the only contact points $h(t)$ and $H(t;h)$ have are $b_i$, $i=1,\dots,9$. Therefore, using Proposition \ref{prop23} we get that for any $x\in \mathbb{S}^3$
\[ U_h(x,C_0)=\sum_{y\in C_0}h(x \cdot y)\leq \sum_{y\in C_0} H(x\cdot y;h) =U_H (x,C_0)=120H_0 =U_H (y,C_0),\]
where $y\in C_0$. However, as $H(b_i;h)=h(b_i)$, we have $U_H  (y,C_0)=U_h(y,C_0)$, which implies that $U_h(x,C_0)$ attains its maximum at a point of the code. Therefore, the min-max polarization bound \eqref{600cellBound} is attained at each of the $600$-cell vertices. 

Let $z\in \mathbb{S}^3\setminus C_0$ be a point where the maximum of $U_h(x,C_0)$ is attained. Recall that $T(z,C_0)$ denotes the collection of all inner products from $z$ to points from $C_0$. Select a fixed vertex $w\in C_0$. Since
\begin{eqnarray*}
U_h(z,C_0) &=& \sum_{y\in C_0}h( z\cdot y)\leq \sum_{y\in C_0} H(z\cdot y;h) \\
&=& 120H_0=\sum_{y\in C_0} H(w\cdot y;h)= \sum_{y\in C_0}h( w\cdot y) \\
&=& {\mathcal R}_h (C_0),
\end{eqnarray*}
we must have $T(z,C_0) \subset \{ b_1,\dots,b_8\}$. 
Let 
\[G(t):=g_{15}(t)-\frac{(g_{15}(t))_{12}}{(g_{16})_{12}} g_{16}(t) = \sum_{i=0}^{16} G_{i}P_i^{(4)}(t) .\]
Then $G(t)$ is annihilating polynomial on $T(z,C_0 )$, $G(t)\geq 0$, $G(t)\not\equiv 0$, and $G_{12}=0$. Therefore, applying Proposition \ref{prop23} we obtain
\[0=\sum_{y\in C_0}G(z\cdot y) =\sum_{i=0}^{16} G_i \sum_{y\in C_0}P_i^{(4)}(z\cdot y)=120G_0=\frac{240}{\pi}\int_{-1}^1 G(t)\sqrt{1-t^2}\,dt>0. \]
The resulting contradiction implies that the only maxima of $U_h(x,C_0)$ occur at points of $C_0$. 

The universal optimality of $C_0$ for min-max polarization follows closely the approach in  \cite[Theorem 2.5]{B}. Namely, we know from the universal optimality \textcolor{black}{(for absolutely monotone potential of order $16$)} of the $600$-cell for energy, that for any other code $C$
\[ \sum_{x\not=y \in C_0} h(x\cdot y) \leq  \sum_{x\not=y \in C} h(x\cdot y) \]
Therefore, with fixed $y\in C_0$ we have
\begin{align} \label{PUUB600}{\mathcal R}_{h}(C_0)&=U_h(y,C_0)=\frac{1}{120}\sum_{y\in C_0} U_h(y,C_0)=\frac{1}{120}\sum_{x\not=y \in C_0} h(x\cdot y)+h(1)\nonumber \\
&\leq \frac{1}{120}\sum_{x\not=y \in C} h(x\cdot y)+h(1)=\frac{1}{120}\sum_{x\in C}U_h(x,C)\leq {\mathcal R}_{h} (C).
\end{align}
Observe that equality for the energy in \eqref{PUUB600} holds if and only if $C=C_0$. This concludes the proof. 
\end {proof}
 
\section{Some spherical designs with optimal max-min polarization}\label{max-min attained}

We note that the max-min polarization (also referred to as maximal polarization) is established in only few cases. In this section we shall combine our techniques from Sections \ref{LowerBound} and \ref{UpperBound} to derive the max-min polarization in some special examples, such as when $N\leq n$, when $N=n+1$, and when $N=2n$. We note also that in these case our bounds are attained. 
\begin{example} \label{BHS14_2_1} First, we will utilize our PULB and PUUB bounds to provide an alternative proof of \cite[Proposition 14.2.1]{BHS}, namely that for a potential $h(t)$ with positive first and second derivatives, the maximal polarization for $2\leq N\leq n$ occurs for spherical $1$-designs. We note that proofs of Theorems \ref{PULB} and \ref{PUUB} and Corollaries \ref{PULB-N} and \ref{PUUB-N} can be modified in this case (with $\tau=1$). The inequalities \eqref{PolarizationULB} and \eqref{PULB_N} yield
\[ \mathcal{Q}^{(h)}(N)\geq \mathcal{Q}^{(h)}(n,1,N)\geq Nh(0),\]
obtained by using the optimal (sub-potential) linear polynomial $f(t)=h^\prime (0) t + h(0)$. Indeed, for any code $C\subset \mathbb{S}^{n-1}$ we have
\[U_h(x,C)=\sum_{y\in C} h(x\cdot y )\geq \sum_{y\in C} f( x\cdot y )=h^\prime (0)\sum_{y\in C} (x\cdot y) +Nh(0).\]
\begin{figure}[htbp]
\centering
\includegraphics[width=3.5 in]{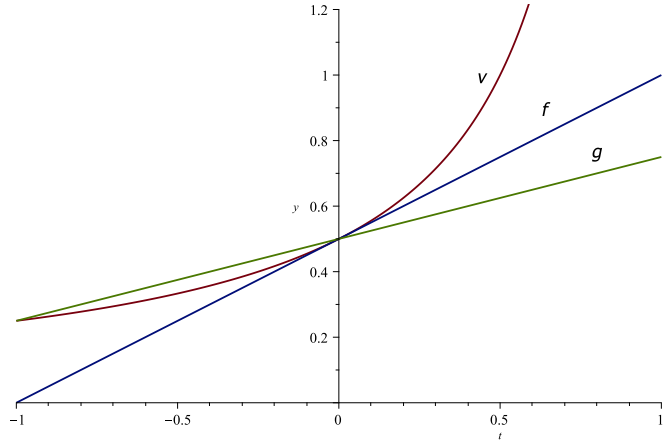}
\caption{Graphs of Newton potential ${\rm v}(t)=\frac{1}{2(1-t)},\ f(t)=t/2+1/2$, and $g(t)=t/4+1/2$.}
\label{fig:7}
\end{figure}
On the other hand any code $C\subset \mathbb{S}^{n-1}$ with $|C|=N\leq n$ points will be contained in a hemisphere and there will be points $x$ such that $T(x,C)\in[-1,0]$. With $g(t):=(h(0)-h(-1))t+h(0)$ (see Figure \ref{fig:7}) we have $g(t)\le h(t)$ for $t\in [-1,0]$ and so
\begin{align*} Q_h(C) &\leq \min_{\{x: T(x,C)\subset [-1,0]\} } U_h(x,C)\leq \min_{\{x: T(x,C)\subset [-1,0]\} }U_g(x,C)\\
&=  Nh(0) +(h(0)-h(-1))\min_{\{x: T(x,C)\subset [-1,0]\} } \sum_{y\in C} ( x \cdot y) \leq Nh(0).
\end{align*}
Thus, $ \mathcal{Q}^{(h)}(N)\leq Nh(0)$, which shows that equality holds. The latter is only possible when $ \mathcal{Q}^{(h)}(N)= \mathcal{Q}^{(h)}(n,1,N)$, or the code is a $1$-design.\hfill $\Box$
\end{example}

\begin{example}\label{simplex}
We should note that the maximal polarization property of the regular simplex is quite a complex problem that was only recently established by Su \cite{Su} for the case $n=3$ and for general $n$ by Borodachov in \cite{B2} \textcolor{black}{(see also \cite{NR2})}. We shall use our bounds to derive an alternative proof of this fact. Denote with $C_0$ the regular simplex on $\mathbb{S}^{n-1}$. We note that the regular simplex is a $2$-design. We also need the fact that the best covering radius for $n+1$ points on $\mathbb{S}^{n-1}$ is obtained when the points form a regular simplex (\cite[Theorem 6.5.1, p. 189]{BoJr}). In our notation given in \eqref{CoveringRadius} and \eqref{CoveringRadiusDes} we have that
\[s_N=s_{n+1}=s_{2,N}=\frac{1}{n}.\]
Given an absolutely monotone potential $h(t)$ of order $3$, the nodes and weights quadrature parameters in Theorem \ref{PULB} are computed to be $\alpha_0=-1,\ \alpha_1=1/n$, and $\rho_0=1/(n+1), \rho_1 =n/(n+1)$, respectively, which yields that 
\begin{equation}\label{simplex_bound1} \mathcal{Q}^{(h)}(n,2,n+1)\geq (n+1)\left[ \rho_0h(\alpha_0)+\rho_1(\alpha_1)\right]=h(-1)+nh\left( \frac{1}{n} \right).\end{equation}

On the other hand, since $s_{n+1}=1/n$, for any code $C$ on $\mathbb{S}^{n-1}$ of cardinality $|C|=n+1$ there exists $x$ such that $T(x,C)\subset [-1,s_{n+1}]$. In fact, the linear polynomial interpolating $h(t)$ at $t=-1, 1/n$
\begin{eqnarray*}
g(t) &:=& \frac{n(h(1/n)-h(-1))}{n+1}\left(t-\frac{1}{n}\right)+h\left(\frac{1}{n}\right) \\
&=& \frac{n(h(1/n)-h(-1))}{n+1}t +\frac{1}{n+1}h(-1)+\frac{n}{n+1}h\left(\frac{1}{n}\right)
\end{eqnarray*} 
stays above $h(t)$ on $[-1,1/n]$, and hence for any code $C\subset \mathbb{S}^{n-1}$, $|C|=n+1$ we have that 
\begin{eqnarray} \label{SimplexUpper}Q_h(C) &\leq& \min_{\{x: T(x,C)\subset [-1,1/n]\} } U_h(x,C)\leq \min_{\{x: T(x,C)\subset [-1,1/n]\} }U_g(x,C)\nonumber\\
&=& h(-1)+nh(1/n)+ n(h(1/n)-h(-1)) \left( \min_{\{x: T(x,C)\subset [-1,1/n]\} } ( x\cdot y_C ) \right) \\
&\leq& h(-1)+nh(1/n), \nonumber
\end{eqnarray}
where $y_C=(\sum_{y\in C}y)/(n+1)$ denotes the centroid of $C$. In the last inequality we used Lemma  \ref{SimplexBound} below. Therefore, when $h$ is absolutely monotone potential of order $3$, $\mathcal{Q}^{(h)}(n+1)=h(-1)+nh(1/n)$ with equality attained only for the regular simplex and minima located at its antipodal points (see \cite[Theorem 2.1]{B2}).
\end{example}


\begin{lemma}\label{SimplexBound}
For any spherical code $C\subset \mathbb{S}^{n-1}$, $|C|=n+1$, the bound 
\[\min_{\{x: T(x,C)\subset [-1,1/n]\} }  ( x\cdot y_C ) \leq 0\] holds.
\end{lemma}
\begin{proof} 
If the code $C=\{ x_0,\dots, x_n\}$ is contained in a hyperspace, then clearly there is a point $x\in \mathbb{S}^{n-1}$, such that $( x\cdot x_i )\leq 0$ for all $i=0,\dots,n$ and the conclusion is obvious. 

Hence, we may assume that $C$ is in general position (not on a hyperplane) and that $0\in {\rm int}(C)$. Let $\{b_0, \dots,b_n\}$ be the origin's barycentric coordinates, namely the unique positive numbers such that $0=b_0x_0+\dots+b_nx_n$, and $b_0+\dots+b_n=1$. Clearly, there is a barycentric coordinate, say $b_0\leq 1/(n+1)$. Let $H_0$ be the hyperplane determined by $\{ x_1,\dots,x_n\}$, $r_0:={\rm dist}(0,H)$, and $a_0:={\rm dist}(x_0,H)$. Since $0$ is in the interior of the convex hull of $C$ we have that $x_0$ and $0$ are in the same half-space relative to $H_0$. Then  \cite[Lemma 3.2]{B2} yields that $nr_0\leq a_0$ and $r_0\leq 1/n$.

Without loss of generality we may assume that $H$ is orthogonal to the South-North polar axis of $\mathbb{S}^{n-1}$ and located in the "Southern hemisphere". Let $\widetilde{x}$ be the South pole. Then $( \widetilde{x}\cdot x_i )=r_0\leq 1/n$ for $i=1,\dots,n$ and  $( \widetilde{x}\cdot x_0 )=-a_0$. Therefore, $T(\widetilde{x},C) \subseteq [-1,1/n]$ and $( \widetilde{x}\cdot y_C )\leq 0$, which concludes the proof.
\end{proof}

\begin{example}\label{simplex_Negative}
In this example we utilize Theorem \ref{PULB_Negative} to show that the simplex is universally optimal configuration with respect to potentials $h$ that are absolutely monotone of order $2$ with $h^{(3)}(t)\leq 0$, $t\in (-1,1)$.
This is a counterpart to \cite[Theorem 2.2]{B2}. The quadrature nodes in this case are $\alpha_1=-1/n$ and $\alpha_2=1$, while the weights are $\rho_1=n/(n+1)$ and $\rho_2=1/(n+1)$. The PULB  \eqref{PolarizationULB_Negative} yields
\begin{equation}\label{simplex_bound2} \mathcal{Q}^{(h)}(n,2,n+1)\geq (n+1)\left[ \rho_0 h(\alpha_0)+\rho_1(\alpha_1)\right]=h(1)+nh\left( -\frac{1}{n} \right).\end{equation}

To derive matching \eqref{simplex_bound2} upper bound, we first consider codes $C$, $|C|=n+1$, that are situated in one half-sphere, say of nonpositive altitude, we can use the monotonicity and convexity to estimate the value of the potential $h$ at the North pole $\widetilde{x}$ to derive that 
\[ Q_h(C)\leq U_h (\widetilde{x},C)\leq (n+1)h(0)\leq nh(-1/n)+h(1).\]

Therefore, we can restrict ourselves to simplexes  $C$ that contain the origin in the interior of their convex hull. Suppose there is a facet determined by $n$ vertices, say $\{x_1,\dots,x_n\}$, whose distance to the origin is $r_0>1/n$. Orienting the facet to be perpendicular to the South-North polar axis at an altitude $-r_0$ ($\widetilde{x}$ being the North pole again) we obtain the estimate
\[Q_h(C)\leq U_h(\widetilde{x},C)=nh(-r_0)+h(\widetilde{x}\cdot x_0)\leq nh(-1/n)+h(1).\]

Therefore, we need only to consider codes $C$, $|C|=n+1$, such that {\em all} $n+1$ facets are within $1/n$ distance from the origin and that the origin is in the convex hull of $C$. Let $b_0,\dots, b_n$ be the barycentric coordinates of the origin. As it is  they are all positive and add up to $1$, at least one, say $b_0\geq 1/(n+1)$. As done above, orient the facet $\{ x_1,\dots,x_n \}$ to be perpendicular to the South-North polar axis at altitude $- r_0$ and let $\widetilde{x}$ be the North pole. Then $T(\widetilde{x},C)\subset [-1/n,1]$. Since $b_0 x_0+\dots +b_n x_n=0$, we  have 
\[ 0=b_0 \widetilde{x}\cdot x_0 \cdot +\dots + b_n \widetilde{x}\cdot x_n = b_0\widetilde{x}\cdot x_0 -(1-b_0)r_0\geq \frac{1}{n+1}\widetilde{x}\cdot x_0- \frac{n}{n+1}r_0 =  \widetilde{x}\cdot y_C,\] 
where $y_C$ is the centroid of $C$.
Denote the linear polynomial that interpolates $h(t)$ at $t=-1/n$ and $t=1$ with 
\[g^*(t) := \frac{n(h(1)-h(-1/n))}{n+1}\left(t+\frac{1}{n}\right)+h\left(-\frac{1}{n}\right).\]  
We modify the sequence of inequalities \eqref{SimplexUpper} as follows
\begin{eqnarray} \label{SimplexUpper2}Q_h(C) &\leq&U_h(\widetilde{x},C)\leq U_{g^*}(\widetilde{x},C)\nonumber\\
&=& h(1)+nh(-1/n)+ n(h(1)-h(-1/n)) ( \widetilde{x}\cdot y_C ) \\
&\leq& h(1)+nh(-1/n)\nonumber
\end{eqnarray}
Equality may only hold if $b_i=1/(n+1)$, $i=0,\dots,n$ and if all inner products between distinct points in the code are equal to $-1/n$, which determines the simplex uniquely. 
\hfill $\Box$
\end{example}

\begin{example} \label{Cross} We next consider the cross-polytope $C_{2n}$ of $2n$ points on $\mathbb{S}^{n-1}$ (see \cite{NR2} for the case of power potentials). This is a $3$-design and the bound \eqref{PULB-N} yields that 
\[ \mathcal{Q}^{(h)}(n,3,2n)\geq 2n \left[ \rho_0h(\alpha_0)+\rho_1(\alpha_1)\right]= nh\left( -\frac{1}{\sqrt{n}} \right) +  
n h \left( \frac{1}{\sqrt{n}} \right).\]
Indeed, the corresponding Gegenbauer polynomial $P_2^{(n)}(t)=(nt^2-1)/(n-1)$, has zeros $\pm 1/\sqrt{n}$ and the weights are determined as $\rho_0=\rho_1=1/2$. The optimal cubic polynomial in the class $\mathcal{L}(n,3,h)$ is the Hermite interpolant to $h(t)$ at $\pm 1/\sqrt{n}$. For the $2^n$ points $\widetilde{x}=(\pm 1/\sqrt{n},\dots,\pm 1/\sqrt{n})$ we have that $ T(\widetilde{x},C_{2n})=\pm 1/\sqrt{n}$ and 
\begin{equation}\label{CrossPoly0}
\mathcal{Q}_h(C_{2n})=\min_{x\in \mathbb{S}^{n-1}} U_h(x,C_{2n})= U_h(\widetilde{x},C_{2n}) = n h(-1/\sqrt{n})+n h(1/\sqrt{n}).\end{equation}
Corollary \ref{PULB-N} yields that 
\begin{equation}\label{CrossPoly1}
\mathcal{Q}^{(h)}(2n)\geq  n h(-1/\sqrt{n})+n h(1/\sqrt{n}).
\end{equation}

For an upper bound we shall make use of the quadratic polynomial  $g(t)$ that interpolates the potential function and its derivative at $-1/\sqrt{n}$ and (only) interpolates the function at $1/\sqrt{n}$. As in the proof of Theorem \ref{PUUB}, one derives that $g(t)\in \mathcal{A}(n,1/\sqrt{n},3,h)$ (note that in Theorem \ref{PUUB} we considered the polynomial of degree $3$ that additionally interpolates $h(t)$ at $-1$). We find that 
 \[ g(t)=h\left(-\frac{1}{\sqrt{n}}\right)\frac{(\sqrt{n}t+3)(1-\sqrt{n}t)}{4}+h^\prime \left(-\frac{1}{\sqrt{n}}\right)\frac{1-nt^2}{2\sqrt{n}}+h\left(\frac{1}{\sqrt{n}}\right)\left( \frac{\sqrt{n}t+1}{2}\right)^2.\]
 Its Gegenbauer expansion is 
 \begin{align} \label{CrossPolyGeg}
 g(t)&=\frac{n-1}{4}\left[ h\left(\frac{1}{\sqrt{n}}\right)-h\left(-\frac{1}{\sqrt{n}}\right)-\frac{2}{\sqrt{n}}h^\prime\left(-\frac{1}{\sqrt{n}}\right)\right]P_2^{(n)}(t)\nonumber\\
 &+\frac{\sqrt{n}}{2}
 \left[ h\left(\frac{1}{\sqrt{n}}\right)-h\left(-\frac{1}{\sqrt{n}}\right)\right] P_1^{(n)}(t) \\ 
 &+\frac{1}{2}h\left( -\frac{1}{\sqrt{n}} \right)  +  \frac{1}{2} h \left( \frac{1}{\sqrt{n}} \right)
= g_2 P_2^{(n)}(t)+ g_1P_1^{(n)}(t)+g_0. \nonumber
 \end{align}
Note that for increasing strictly convex functions $h(t)$ the Gegenbauer coefficients $g_1$ and $g_2$ are positive. 

The best covering radius for $2n$ points on $S^{n-1}$ is known for $n=2,3,4$ (see \cite[Section 6.7]{BoJr} for $n=4$), and for $n\geq 5$ it is conjectured \cite[Conjecture 6.7.3]{BoJr} that 
\[ s_{2n} = s_{2,2n} = s_{3,2n} = \frac{1}{\sqrt{n}} .\]
For the rest of this \textcolor{black}{proof} we shall assume this conjecture to be true. Let $C\subset \mathbb{S}^{n-1}$ be an arbitrary code with $|C|=2n$. The facets of $C$ form spherical caps (with no points of $C$ in the interior of the spherical cap), let us select one with a largest diameter and let $\widetilde{x}$ be its center. From the best covering radius conjecture we have that $T(\widetilde{x},C)\subset[-1,1/\sqrt{n}]$. 
\begin{align*} Q_h(C) &\leq \min_{\{x: T(x,C)\subset [-1,1/\sqrt{n}]\} } U_h(x,C)\leq \min_{\{x: T(x,C)\subset [-1,1/\sqrt{n}]\} }U_g(x,C).
\end{align*}
If $C$ is a spherical $2$-design, then from \eqref{CrossPolyGeg} we derive that 
\[\min_{\{x: T(x,C)\subset [-1,1/\sqrt{n}]\} }U_g(x,C)=nh\left( -\frac{1}{\sqrt{n}} \right) +  
n h \left( \frac{1}{\sqrt{n}} \right),\]
which shows that the cross-polytope solves the maximal polarization problem among all spherical $2$-designs, which we formulate as a separate proposition.\hfill $\Box$
\end{example}

\begin{proposition}\label{CrossPolytope} For all $n=2,3,4$ the cross polytope has optimal max-min polarization among all spherical $2$-designs. Should \cite[Conjecture 6.7.3]{BoJr} hold {\rm(}or $s_{2n}=1/\sqrt{n}${\rm)}, the proposition is true for all $n$. 
\end{proposition}

We are able to prove that the cross polytope has optimal max-min polarization among the class of {\em centered codes}, namely codes for which $T(\widetilde{x},C)\subset[-1/\sqrt{n},1/\sqrt{n}]$ for some $\widetilde{x}\in \mathbb{S}^{n-1}$.

\begin{proposition}\label{centralized}
For all $n\in \mathbb{N}$ the cross polytope is optimal for the max-min polarization problem among all centered codes of cardinality $2n$.  Moreover, for any optimal centered code $C^*$, there exist a point $\widetilde{x} \in \mathbb{S}^{n-1}$ such that $T(\widetilde{x},C^*)=\{-1/\sqrt{n},1/\sqrt{n}\}$ with each inner product value having equal multiplicity $n$.
\end{proposition}
\begin{proof}
Let $\widetilde{x}\in \mathbb{S}^{n-1}$ be such that $T(\widetilde{x},C)\subset[-1/\sqrt{n},1/\sqrt{n}]$. Observe, that $T(-\widetilde{x},C)\subset[-1/\sqrt{n},1/\sqrt{n}]$ as well. We will prove that for such centralized codes we have
\begin{equation}\label{CrossPoly2}
\mathcal{Q}_h (C)\leq  \min\left\{ U_h(-\widetilde{x},C),U_h(\widetilde{x},C)\right\}\leq n h(-1/\sqrt{n})+n h(1/\sqrt{n}).
\end{equation}
 Since $h(t)\leq g(t)$ on $[-1/\sqrt{n},1/\sqrt{n}]$, we have
\begin{equation}\label{CrossPoly3}
U_h(\widetilde{x},C)\leq U_g(\widetilde{x},C),\quad U_h(-\widetilde{x},C)\leq U_g(-\widetilde{x},C).\end{equation}
Using $P_1^{(n)}(t)=t$ and $P_2^{(n)}(t)=(nt^2-1)/(n-1)$, we compute
\begin{align} 
U_g(\widetilde{x},C)&=g_2\sum_{y\in C}P_2^{(n)}(\widetilde{x}\cdot y) +g_1\sum_{y\in C}P_1^{(n)}(\widetilde{x}\cdot y)+n h(-1/\sqrt{n})+n h(1/\sqrt{n})\nonumber \\
&=g_2\sum_{y\in C}\frac{n(\widetilde{x}\cdot y)^2-1}{n-1}+g_1\sum_{y\in C}(\widetilde{x}\cdot y)+n h(-1/\sqrt{n})+n h(1/\sqrt{n})\nonumber\\
&=g_2\sum_{y\in C}\frac{n(\widetilde{x}\cdot y)^2-1}{n-1}+(2n g_1)(\widetilde{x}\cdot y_C)+n h(-1/\sqrt{n})+n h(1/\sqrt{n}), \nonumber
\end{align}
where $y_C$ is the center of mass of $C$. Similarly,
\[U_g(-\widetilde{x},C)=g_2\sum_{y\in C}\frac{n(\widetilde{x}\cdot y)^2-1}{n-1}-(2n g_1)(\widetilde{x}\cdot y_C)+n h(-1/\sqrt{n})+n h(1/\sqrt{n}).\]
Since $T(\widetilde{x},C)\subset [-1/\sqrt{n},1/\sqrt{n}]$ and $T(-\widetilde{x},C)\subset [-1/\sqrt{n},1/\sqrt{n}]$, we conclude
\[ \min\left\{ U_g(-\widetilde{x},C),U_g(\widetilde{x},C)\right\}\leq n h(-1/\sqrt{n})+n h(1/\sqrt{n}) , \]
which along with \eqref{CrossPoly3} implies \eqref{CrossPoly2}. Obviously, the cross-polytope $C_{2n}$ is centered code, so equation \eqref{CrossPoly0} implies that $C_{2n}$ is an optimal code for the max-min polarization.

Next, let $C^*$ be an optimal max-min polarization code among all centered codes with cardinality $2n$. As $C_{2n}$ belongs to this class we have 
\[ n h(-1/\sqrt{n})+n h(1/\sqrt{n})= {\mathcal Q}_h(C_{2n}) = {\mathcal Q}_h(C^*) \leq \min\left\{ U_h(-\widetilde{x},C^*),U_h(\widetilde{x},C^*)\right\}.\]
From \eqref{CrossPoly3} we derive that 
\begin{align*} 
n h(-1/\sqrt{n})+n h(1/\sqrt{n}) &\leq \min\left\{ U_g(-\widetilde{x},C^*),U_g(\widetilde{x},C^*)\right\} \\
&= g_2\sum_{y\in C^*}\frac{n(\widetilde{x}\cdot y)^2-1}{n-1}-2n g_1|\widetilde{x}\cdot y_{C^*}|+n h(-1/\sqrt{n})+n h(1/\sqrt{n}),
\end{align*}
which implies that $T(\widetilde{x},C^*)=\{ -1/\sqrt{n},1/\sqrt{n}\}$ and $\widetilde{x}\cdot y_{C^*}=0$. Hence, there are $n$ inner products of each of the two values. 

Therefore, any optimal $C^*$ splits into a disjoint union of two codes $C_1^*$ and $C_2^*$, each of cardinality $n$, such that 
\[ C_1^*:=\{y_1,\dots,y_n\}, \quad C_2^*:=\{z_1,\dots,z_n\},  \quad \widetilde{x}\cdot y_i=-\widetilde{x}\cdot z_j=\frac{1}{\sqrt{n}},\ i,j=1,\dots,n.\]

\end{proof}
%

 \noindent{\bf Acknowledgments.}  This material is based upon work supported by the National Science Foundation under Grant No. DMS-1929284 while the authors were in residence at the Institute for Computational and Experimental Research in Mathematics in Providence, RI, during the {\it Colaborate@ICERM} program. The research of the first author was supported, in part, by Ministry of Education and Science of Bulgaria under Grant no. DO1-387/18.12.2020 “National Centre for High-Performance and Distributed Computing”. The research of the second author was supported, in part, by NSF grant DMS-1936543.
The research of the fifth author was supported, in part, by Bulgarian NSF grant KP-06-N32/2-2019.

\end{document}